\documentclass[
a4paper,
reqno,
11pt,
dvipdfm
]{article}
\usepackage{graphicx}
\usepackage{
amssymb,
amsmath,
amsthm,
eucal,
dsfont,
tikz}
\usepackage[margin=2cm]{caption}
\makeatletter
 
  \@addtoreset{equation}{section}
 \makeatother
\theoremstyle{plain}

\newtheorem{theorem}{Theorem}[section]
\newtheorem{lemma}[theorem]{Lemma}
\newtheorem{proposition}[theorem]{Proposition}
\newtheorem{corollary}[theorem]{Corollary}
\theoremstyle{definition}
\newtheorem{definition}[theorem]{Definition}
\newtheorem{remark}[theorem]{Remark}
\newtheorem{example}[theorem]{Example}
\newtheorem{assumption}[theorem]{Assumption}

\newcommand{\norm}[1]{{||#1||}}

\newcommand{\wtilde}[1]{{\widetilde{#1}}}

\def\Id{\mathop{\mathrm{Id}}\nolimits}

\def\Ran{\mathop{\mathrm{Ran}}\nolimits}
\def\Ker{\mathop{\mathrm{Ker}}\nolimits}
\def\codim{\mathop{\mathrm{codim}}\nolimits}

\def\Re{\mathop{\mathrm{Re}}\nolimits}
\def\Im{\mathop{\mathrm{Im}}\nolimits}

\def\sgn{\mathop{\mathrm{sgn}}\nolimits}
\def\dist{\mathop{\mathrm{dist}}\nolimits}
\def\Hardy{\mathop{\mathrm{H}}\nolimits}
\def\Num{\mathop{\mathrm{Num}}\nolimits}

\def\ess{\mathop{\mathrm{ess}}\nolimits}

\def\R{{\mathbb{R}}}

\def\C{{\mathbb{C}}}

\def\H{{\mathcal{H}}}
\def\A{{\mathcal{A}}}
\def\B{{\mathcal{B}}}

\def\<{{\langle}}
\def\>{{\rangle}}

\def\hookr{{\hookrightarrow}}
\def\ep{{\varepsilon}}

\title
{Eigenvalue bounds for non-self-adjoint Schr\"odinger operators with the inverse-square potential}
\author{Haruya Mizutani
}

\AtEndDocument{\bigskip{\footnotesize\textsc{Department of Mathematics, Graduate School of Science, Osaka University, Toyonaka, Osaka 560-0043, Japan.} \par\textit{E-mail address}: \texttt{haruya@math.sci.osaka-u.ac.jp} \par}}

\date{\empty}

\begin{document}
\maketitle

\begin{abstract}
The purpose of this paper is to study spectral properties of non-self-adjoint Schr\"odinger operators $-\Delta-\frac{(n-2)^2}{4|x|^{2}}+V$ on $\R^n$ with complex-valued potentials $V\in L^{p,\infty}$, $p>n/2$. We prove Keller type inequalities which measure the radius of a disc containing the discrete spectrum, in terms of the $L^{p,\infty}$ norm of $V$. Similar inequalities also hold if the inverse-square potential is replaced by a large class of subcritical potentials with critical singularities at the origin. The  main new ingredient in the proof is the uniform Sobolev inequality of Kenig-Ruiz-Sogge type for Schr\"odinger operators with strongly singular potentials, which is of independent interest. 
\end{abstract}


\section{Introduction and Main results}
\label{section_Introduction}

This paper is concerned with spectral properties of the non-self-adjoint Schr\"odinger operator $$-\Delta-\frac{(n-2)^2}{4|x|^{2}}+V$$ on $\R^n$, $n\ge3$, with a complex-valued potential $V:\R^n\to\C$. In particular we are interested in the Keller type inequality for eigenvalues $E\in \C\setminus[0,\infty)$ of the form
\begin{align}
\label{KLT}
|E|^\gamma \le C_{\gamma,n}\int |V(x)|^{n/2+\gamma} dx
\end{align}
for some $\gamma>0$ and $C_{\gamma,n}>0$ independent of $E$ and $V$. This estimate  gives a quantitative bound 
of the radius of a disk in $\C$ which contains the discrete spectrum $\sigma_{\mathrm{d}}(-\Delta-\frac{(n-2)^2}{4|x|^{2}}+V)$. 

In the case of $-\Delta+V$ with real-valued potential $V$, the estimate \eqref{KLT}  for the lowest negative eigenvalue $E$ with $|V(x)|$ replaced by $V_-(x)$, the negative part of $V(x)$, was first found by Keller \cite{Kel} for $n=1$, $\gamma\ge1/2$ and, later,  generalized to a much stronger inequality  known as the Lieb-Thirring inequality: 
\begin{align}\label{LT}\sum_{E\in \sigma_{\mathrm{d}}(-\Delta+V)}|E|^\gamma\le L_{\gamma,n}\int V_-(x)^{n/2+\gamma}dx,\end{align} where $\gamma\ge0$ if $n\ge3$, $\gamma>0$ if $n=2$ and $\gamma\ge1/2$ if $n=1$ (see \cite{LiTh,LiSe}). 

In the last decade, numerous improvements and generalizations of these two inequalities \eqref{KLT} and \eqref{LT} have been made. On one hand, \eqref{KLT} has been extended to the case $-\Delta+V$ with complex valued potentials by \cite{AAD,Fra2,FrSi,Fra3}. More precisely, corresponding results proved in these papers are summarized as follows: 
\begin{itemize}
\item[(i)] \eqref{KLT} holds for $\gamma=1/2$ if $n=1$ (\cite{AAD}) and for $0<\gamma\le 1/2$ if $n\ge2$ (\cite{Fra2}). Moreover, if $n\ge2$ and $0<\gamma<1/2$ then $\norm{V}_{L^{n/2+\gamma}}$ can be replaced by the Morrey-Campanato norm $\norm{V}_{\mathcal L^{\alpha,n/\alpha}}$ with $\alpha=2n/(n+2\gamma)$ (see \cite[Theorem 3]{Fra2}).
\item[(ii)] \eqref{KLT} with $\norm{V}_{L^{n/2+\gamma}}$ is replaced by $\norm{V}_{L^{n/2+\gamma}(\R_+,r^{n-1}dr;L^\infty(\mathbb S^{n-1}))}$ holds for all $0<\gamma<n/2$ and $n\ge2$. In particular, \eqref{KLT} holds for $0<\gamma<n/2$ if $V$ is radially symmetric (\cite{FrSi}). 
\item[(iii)] If $n\ge1$ and $\gamma\ge1/2$ then the following estimate holds (\cite{Fra3}):
$$
|E|^{1/2} (\dist(E,[0,\infty)))^{\gamma-1/2}\le C_{\gamma,n}\int |V(x)|^{n/2+\gamma}dx.
$$
\end{itemize}
In general, Laptev-Safronov \cite{LaSa} conjectured that \eqref{KLT} holds for $0<\gamma\le n/2$. Note that, in contrast to the self-adjoint case, there is no min-max principle for the non-self-adjoint case and the usual variational argument by \cite{Kel} does not work well. Instead, above results (i)--(iii) rely on the Birman-Schwinger principle, together with precise estimates for the Birman-Schwinger kernel $\sgn V|V|^{1/2}(-\Delta-z)^{-1}|V|^{1/2}$ in terms of $z$ and the $L^{n/2+\gamma}$-norm of $V$. In particular, uniform Sobolev inequalities for the free resolvent $(-\Delta-z)^{-1}$ due to Kenig-Ruiz-Sogge \cite{KRS} were used in \cite{Fra2,Fra3,FrSi}. Lieb-Thirring type inequalities for the moment of  eigenvalues in the non-self-adjoint case have also been extensively studied. In addition to the above papers, we refer to  \cite{FLLS,LaSa,DHK,FrSa} and references therein. %

On the other hand, it is known that the operator $-\Delta +V$ with a sufficiently small potential $V$ has no negative eigenvalue if $n\ge3$. This can be seen from, {\it e.g.}, Hardy's inequality 
\begin{align}
\label{Hardy}
C_{\Hardy}\int  |x|^{-2}|u|^2dx\le \int|\nabla u|^2dx,\quad u\in C_0^\infty(\R^n),\quad C_{\Hardy}:=\frac{(n-2)^2}{4}. 
\end{align}
Hence it is natural to ask which part of the potential is responsible for negative eigenvalues. In this context, Ekholm-Frank \cite{EkFr} proved the so-called  Hardy-Lieb-Thirring inequality: 
\begin{align}
\label{HLT}
\sum_{E\in \sigma_{\mathrm{d}}(-\Delta-C_{\Hardy}|x|^{-2}+V)}(-E)^\gamma\le L'_{\gamma,n}\int V(x)_-^{\frac n2+\gamma}dx,\quad n\ge3,\ \gamma>0,
\end{align}
for the operator $-\Delta-C_{\Hardy}|x|^{-2}+V$ with real-valued potential $V$. This inequality suggests that only the part of the potential which is stronger than $-C_{\Hardy}|x|^{-2}$ is responsible for negative eigenvalues. There are also several works on spectral properties of $-\Delta-C_{\Hardy}|x|^{-2}+V$ with real-valued potentials (see \cite{BiLa,Wei,FLS,Fra1} and references therein). 
However, 
the case with complex-valued potentials is less understood. In particular, there seems to be no previous literature on the above type inequalities \eqref{KLT} or \eqref{HLT}. Here note that results by \cite{AAD,Fra2,FrSi,Fra3} cannot be applied to $-\Delta-C_{\Hardy}|x|^{-2}+V$ since $|x|^{-2}\notin L^{p,\infty}(\R^n)$ for any $p\neq n/2$. 

In this paper we focus on bounds for individual eigenvalues and extend (i)--(iii) to the operator $-\Delta-C_{\Hardy}|x|^{-2}+V$. More precisely, we show Keller type eigenvalue bounds \eqref{KLT} with $\norm{V}_{L^{n/2+\gamma}}$ replaced by the weak-Lebesgue norm $\norm{V}_{L^{n/2+\gamma,\infty}}$, which enables us to deal with more singular potentials than in the previous literatures. Furthermore,  for a wide class of subcritical potentials $V_\delta$ (see Assumption \ref{assumption_A} below), 
we also show similar bounds for operators $-\Delta+V_\delta+V$. These results particularly improve a part of previous results (i)--(iii) even in the free case $V_\delta\equiv 0$. 

The proof of these results basically follows the strategy in \cite{Fra2,Fra3}. 
The main new ingredients in the present case are uniform Sobolev inequalities of Kenig-Ruiz-Sogge type for the resolvents $(-\Delta-C_{\Hardy}|x|^{-2}-z)^{-1}$ and $(-\Delta+V_\delta-z)^{-1}$. For the critical case, it takes the form
$$
\norm{(-\Delta-C_{\Hardy}|x|^{-2}-z)^{-1}}_{L^{p,2}\to L^{p',2}}\le C_p|z|^{-\frac{n+2}{2}+\frac np},\quad z\in \C\setminus[0,\infty),
$$
where $2n/(n+2)<p\le 2(n+1)/(n+3)$, $p'=p/(p-1)$ and $L^{p,q}$ is the Lorentz space (see Section \ref{section_Sobolev} for more details). 
Uniform Sobolev inequalities also play an important role in the study of unique continuation properties  (see \cite{KRS}) or the limiting absorption principle (see \cite{GoSc,IoSc}) for Schr\"odinger operators with rough potentials. Therefore, extending them to more general Schr\"odinger operators with critical singularities may be of independent interest. 

\subsection{Main results}
Let us state our main results more precisely. Let $\Delta$ be the Laplacian and $$H_0=-\Delta-C_{\Hardy}|x|^{-2}$$ a Schr\"odinger operator with the inverse-square	 potential in $\R^n$, $n\ge3$. To be more precise, Hardy's inequality \eqref{Hardy} implies that the quadratic form
$$
\tilde q_0(u)=\int \Big(|\nabla u|^2-C_{\Hardy}|x|^{-2}|u|^2\Big)dx
$$ 
is non-negative and closable on $C_0^\infty(\R^n)$. Then $H_0$ is defined as a unique self-adjoint operator corresponding to the closure $q_0$ of $\tilde q_0$. Since $q_0$ is non-negative and rotationally invariant, so is $H_0$. Moreover, one has $\sigma(H_0)=\sigma_{\mathrm{\ess}}(H_0)=[0,\infty)$. Suppose that $V$ is (possibly) complex-valued and that 
$D(H_0^{1/2})\subset D(|V|^{1/2})$ and $|V|^{1/2}(H_0+1)^{-1/2}$ is compact. Then the quadratic form 
$$
q_0(u)+\int V|u|^2dx
$$
is sectorial on $D(q_0)$ and generates an $m$-sectorial operator $H_0+V$. Moreover, we have 
$$
\sigma_{\ess}(H_0+V)=[0,\infty),\quad 
\sigma(H_0+V)=\sigma_{\mathrm{d}}(H_0+V)\cup[0,\infty),\quad
\sigma_{\mathrm{d}}(H_0+V)\cap[0,\infty)=\emptyset,
$$
where $\sigma_{\mathrm{ess}}(A)$ and $\sigma_{\mathrm{d}}(A)$ denote the essential and discrete spectrum of $A$, respectively.  We refer to Appendix \ref{appendix_form_compact} for more details on the definition and basic spectral properties of $H_0+V$. 

 In order to state main results, we further introduce several notation.  The size of the potential $V$ will be measured by the weak-$L^p$ norm $\norm{V}_{L^{p,\infty}(\R^n)}$ 
and also by a mixed norm $$\norm{V}_{L^{p,\infty}_rL^\infty_\omega(\R^n)}:=\big|\big|\norm{V}_{L^\infty(\mathbb S^{n-1})}\big|\big|_{L^{p,\infty}(\R_+,r^{n-1}dr)}.
$$
Note that $L^p\subset L^{p,\infty}$ and $L^{p,\infty}_rL^\infty_\omega\subset L^{p,\infty}$ and that if $V$ is radially symmetric, $\norm{V}_{L^{p,\infty}_rL^\infty_\omega(\R^n)}$ is equivalent to $\norm{V}_{L^{p,\infty}(\R^n)}$. We refer to Appendix \ref{appendix_interpolation} for more details on Lorentz spaces. Let $d(z)=\dist(z,[0,\infty))$ be the distance between $z$ and $[0,\infty)$ 
 and set $a_+:=\max\{a,0\}$. 

Our first result in this paper is as follows:

\begin{theorem}
\label{theorem_1}
Let $n\ge3$ and $\gamma>0$. Suppose that $V\in L^{n/2+\gamma,\infty}(\R^n)$ and $|V|^{1/2}(H_0+1)^{-1/2}$ is compact .  Then any $E\in \sigma_{\mathrm{d}}(H_0+V)$ satisfies 
\begin{align}
\label{theorem_1_1}
|E|^{\min(\gamma,1/2)}d(E)^{(\gamma-1/2)_+}\le C_{\gamma,n}\norm{V}_{L^{n/2+\gamma,\infty}(\R^n)}^{n/2+\gamma}.
\end{align}
Furthermore, if $V\in L^{n/2+\gamma,\infty}_rL^\infty_\omega(\R^n)$, $\ep=0$ when $n\ge4$ and $\ep>0$ when $n=3$, then 
\begin{align}
\label{theorem_1_2}
|E|^{\min\left(\gamma,\frac{n}{n-1}-\ep\right)}d(E)^{\left(\gamma-\frac{n}{n-1}+\ep\right)_+}\le C_{\gamma,n,\ep}\norm{V}_{L^{n/2+\gamma,\infty}_rL^\infty_\omega(\R^n)}^{n/2+\gamma}.
\end{align}
Here $C_{\gamma,n},C_{\gamma,n,\ep}>0$ can be taken uniformly with respect to  $E$ and $V$. 
\end{theorem}

This theorem implies following spectral consequences: 
\begin{itemize}
\item If $0<\gamma\le1/2$, or $0<\gamma<n/(n-1)$ and $V$ is radial, then \eqref{theorem_1_1} and \eqref{theorem_1_2} imply
$$
|E|^\gamma\le C_{\gamma,n}\norm{V}_{L^{n/2+\gamma,\infty}(\R^n)}^{n/2+\gamma}
$$
which provides a quantitative bound of the radius of a disk which contains $\sigma_{\mathrm{d}}(H_0+V)$. In particular, $\sigma_{\mathrm{d}}(H_0+V)$ is a bounded set in $\C$ in this case. 
\item On the other hand, the bounds \eqref{theorem_1_1} and \eqref{theorem_1_2} imply that for $\Re E>0$, 
\begin{align*}
|\Im E|&\le C |E|^{-\frac{1/2}{\gamma-1/2}}\norm{V}_{L^{n/2+\gamma,\infty}}^{\frac{n/2+\gamma}{\gamma-1/2}}\quad\text{if $\gamma>1/2$},\\
|\Im E|&\le C_\ep |E|^{-\frac{n/(n-1)-\ep}{\gamma-n/(n-1)+\ep}}\norm{V}_{L^{n/2+\gamma,\infty}}^{\frac{n/2+\gamma}{\gamma-n/(n-1)+\ep}}\quad\text{if $\gamma\ge \frac{n}{n-1}$ and $V$ is radial}. 
\end{align*}
This shows that, for any sequence $\{E_j\}\subset \sigma_{\mathrm{d}}(H_0+V)$,  if $\Re E_j\to+\infty$ then $|\Im E_j|\to 0$. Note that $\Re E$ is bounded from below for any $\gamma>0$. \end{itemize}
Note that both properties are shaper than a general fact for $m$-sectorial operators that $\sigma_{\mathrm{d}}(H_0+V)\subset\{z\in \C\ |\ |\arg(z-c)|\le\theta\}$ for some $c\in \R$, $\theta\in[0,\pi/2)$ (see Appendix \ref{appendix_form_compact}). 

The compactness of $|V|^{1/2}(H_0+1)^{-1/2}$ is not a strong restriction, as seen below: 
\begin{example}
\label{example_0}
Define a subspace $L^{p,\infty}_0\subset L^{p,\infty}$ by
$$
L^{p,\infty}_0(\R^n):=\{f\in L^{p,\infty}(\R^n)\ |\ \lim\limits_{R\to\infty}\norm{\mathds 1_{\{|x|\ge R\}}f}_{L^{p,\infty}(\R^n)}=0\},\quad 1\le p\le \infty.
$$
Note that $L^{p,q}\subset L^{p,\infty}_0$ for all $1\le q<\infty$, which can be seen from the facts that $L^{p,q}\subset L^{p,\infty}$ and simple functions are dense in $L^{p,q}$ for any $1\le p,q<\infty$. In particular, $L^p\subset L^{p,\infty}_0$. 
Also note that, if $|V(x)|\le C|x|^{-n/p}$ then $V\in L^{p,\infty}\cap (L^{p,\infty}_0+L^\infty_0)$ but $|x|^{-n/p}\notin L^q$ for any $q\ge1$. We will show in Appendix \ref{appendix_remark} that if $V\in L^{p,\infty}_0(\R^n)+L^\infty_0(\R^n)$ with some $p>\frac  n2$ then both $|V|^{1/2}(1-\Delta)^{-1/2}$ and $|V|^{1/2}(H_0+1)^{-1/2}$ are compact . 
\end{example}

\begin{remark}
\label{remark_0}$ $
\begin{itemize}
\item[(1)] If $V$ is real-valued, $K_{\gamma}(E)=|E|^\gamma$ for any $\gamma$ and \eqref{theorem_1_1} thus implies
\begin{align}
\label{remark_1_1}
|E|^\gamma \le C_{\gamma,n}\norm{V}_{L^{n/2+\gamma,\infty}(\R^n)}^{n/2+\gamma}
\end{align}
for all $\gamma>0$, where one can assume $V\le0$ in this  case by the variational principle. This particularly implies the original version of Keller's inequality for the operator $H_0+V$ in the self-adjoint case. 
\item[(2)] The bound \eqref{theorem_1_1} extends previous results by \cite{Fra2,FrSi,Fra3} to  the operator $H_0+V$. This is not obvious in view of the fact that there is no variational principle in the non-self-adjoint case. Furthermore, if $n=3$ and $V$ is radially symmetric then \eqref{theorem_1_2} implies 
$$
|E|^{\gamma}\le C_{\gamma,n}\norm{V}_{L^{3/2+\gamma,\infty}(\R^3)}^{3/2+\gamma},\quad 0<\gamma<3/2,
$$
which 
proves Laptev-Safronov's conjecture (up to the endpoint $\gamma=3/2$) for $H_0+V$ in the radially symmetric setting. 
\item[(3)] If $n\ge3$ and $\norm{V}_{L^{n/2}}$ is small enough, $-\Delta+V$ has no eigenvalue. This is a consequence of the uniform estimate for $|V|^{1/2}(-\Delta-z)^{-1}V^{1/2}$ with respect to  $z$ (see \cite{Fra2,FrSi}). We also refer to \cite{FKV} in which, among the others, the absence of eigenvalues is studied for more general class of small complex-valued potentials with critical singularity. 
On the other hand, the same property  for $H_0+V$ does not hold. Indeed, when $V<0$, $H_0+V$ has at least one eigenvalue (see \cite{Wei}). Furthermore, it turns out that the resolvent $(H_0-z)^{-1}$ has a logarithmic singularity at $z=0$ (see Remark \ref{remark_2}). 
\item[(4)] Let $\wtilde V=V-C_{\Hardy}|x|^{-2}$. 
Theorem \ref{theorem_1} suggests that only the part of $\wtilde V$ which is stronger than $-C_{\Hardy}|x|^{-2}$ is responsible for non-positive eigenvalues. 
\end{itemize}
\end{remark}

The second result concerns with a subcritical case in the sense that the unperturbed Hamiltonian is bounded from below  by $-\delta\Delta$ with some $\delta>0$. More precisely, we consider a family of potentials $\{V_\delta\}_{\delta>0}$ satisfying the following assumption. 

\begin{assumption}
\label{assumption_A}
(1) $V_\delta(x)$ is real-valued, $|x|V_\delta\in L^{n,\infty}(\R^n)$ and $x\cdot\nabla V_\delta\in L^{n/2,\infty}(\R^n)$. \\
(2) $-\Delta+V_\delta\ge-\delta\Delta$ and $-\Delta-V_\delta-x\cdot\nabla V_\delta\ge-\delta\Delta$ on $C_0^\infty(\R^n)$, \emph{i.e.},
\begin{align}
\label{assumption_A_1}
\int \Big(|\nabla u|^2+V_\delta|u|^2\Big)dx&\ge \delta\int |\nabla u|^2dx,\\
\label{assumption_A_2}
\int \Big(|\nabla u|^2-(V_\delta+x\cdot\nabla V_\delta)|u|^2\Big)dx&\ge \delta\int |\nabla u|^2dx,
\end{align}
hold for all $u\in C_0^\infty(\R^n)$. 
\end{assumption}

The hypothesis $|x|V_\delta\in L^{n,\infty}$ and H\"older's inequality yield $V_\delta\in L^{n/2,\infty}$. Furthermore, it follows from H\"older's and Sobolev's inequalities (see Appendix \ref{appendix_interpolation}) that
\begin{align}
\label{assumption_A_3}
\int |V_\delta||u|^2dx\le C\norm{V_\delta}_{L^{n/2,\infty}}\norm{u}_{L^{{2n}/{(n-2)},2}}^2\le C\norm{V_\delta}_{L^{n/2,\infty}}\norm{\nabla u}_{L^2}^2. 
\end{align}
Thus Assumption \ref{assumption_A} implies that the quadratic form 
$$
\tilde q_\delta(u)=\int \Big(|\nabla u|^2+V_\delta|u|^2\Big)dx,\quad u\in C_0^\infty(\R^n),
$$
is non-negative and closable. Let $H_\delta=-\Delta+V_\delta$ be a unique self-adjoint operator corresponding to the closure of $\tilde q_\delta$. By \eqref{assumption_A_1} and Proposition \ref{proposition_Sobolev_1} (1) in Section \ref{section_Sobolev}, we obtain $\sigma(H_\delta)=\sigma_{\mathrm{ess}}(H_\delta)=[0,\infty)$. Furthermore, it follows from \eqref{assumption_A_1} and \eqref{assumption_A_3} that, given an operator $A$, $A(H_0+1)^{-1/2}$ is compact if and only if $A(1-\Delta)^{-1/2}$ is compact. Given a potential $V:\R^n\to \C$ such that $|V|^{1/2}(H_0+1)^{1/2}$ is compact, one thus can define an $m$-sectorial operator $H_\delta+V$ as in the critical case. Furthermore, $H_\delta+V$ satisfies $\sigma_{\ess}(H_\delta+V)=[0,\infty)$, $\sigma(H_\delta+V)=\sigma_{\mathrm{d}}(H_\delta+V)\cup[0,\infty)$ and $\sigma_{\mathrm{d}}(H_\delta+V)\cap[0,\infty)=\emptyset$ as in the critical case (see Appendix \ref{appendix_form_compact}). 

Then we obtain the same inequalities as in the critical case with constants depending on $\delta$: 

\begin{theorem}
\label{theorem_2}
Let $n\ge3$, $\delta>0$, $\gamma>0$ and let $\ep=0$ if $n\ge4$ or $\ep>0$ if $n=3$. Suppose that $|V|^{1/2}(1-\Delta)^{-1/2}$ is compact. Then there exist $C_{\gamma,n},C_{\gamma,n,\ep}>0$ (independent of $V$ and $\delta$) such that any $E\in \sigma_{\mathrm{d}}(H_\delta+V)$ satisfies 
\begin{align}
\label{theorem_2_1}
|E|^{\min(\gamma,1/2)}d(E)^{(\gamma-1/2)_+}&\le C_{\gamma,n}\delta^{-\min(n+2\gamma,n+1)}\norm{V}_{L^{n/2+\gamma,\infty}(\R^n)}^{n/2+\gamma},\\
\label{theorem_2_2}
|E|^{\min\left(\gamma,\frac{n}{n-1}-\ep\right)}d(E)^{\left(\gamma-\frac{n}{n-1}+\ep\right)_+}&\le C_{\gamma,n,\ep}\delta^{-\min\left(n+2\gamma,\frac{n(n+1)}{n-1}+\ep\right)}\norm{V}_{L^{n/2+\gamma,\infty}_rL^\infty_\omega(\R^n)}^{n/2+\gamma}.
\end{align}
\end{theorem}

\begin{remark}$ $
\begin{itemize}
\item[(1)] Similar spectral consequences as in the critical case also hold for $H_\delta+V$. 
\item[(2)] Theorem \ref{theorem_2} does not imply Theorem \ref{theorem_1} since the right hand sides of \eqref{theorem_2_1} and \eqref{theorem_2_2} blow up 
 as $\delta\to0$. However, the admissible class of $V_\delta$ is much wider than in Theorem \ref{theorem_1}. At first $V_\delta\equiv 0$ obviously satisfies Assumption \ref{assumption_A}. The inverse-square potential $-\sigma |x|^{-2}$ satisfies Assumption \ref{assumption_A} if $\sigma<C_{\Hardy}$. Furthermore, Assumption \ref{assumption_A} is general enough to include several anisotropic potentials such that $|x|^2V_\delta\notin L^\infty$. For instance, let $c_1,c_2>0$, $\alpha\in\R^n$ and $\chi\in C^1(\R)$ so that $0\le \chi\le 1$ and $|\chi^{(k)}(t)|\le |t|^{-k-1}$ for $|t|\ge1$, and define 
$$
V_\delta(x)=(-C_{\Hardy}+c_1)|x|^{-2}-c_2\chi(|x-\alpha|)|x-\alpha|^{-1}. 
$$
Then Assumption \ref{assumption_A} holds with $\delta=c_1-c_2(2+\sup|\chi'|)(|\alpha|+1)$ if 
$$
c_2<\frac{c_1}{(2+\sup|\chi'|)(|\alpha|+1)}. 
$$
One can also consider multiple small Coulomb type singularities. 
\item[(3)] For $\gamma\ge1/2$, Theorem \ref{theorem_2} is new even in the free case $V_\delta\equiv0$ compared with previous results by \cite{Fra2,Fra3} where only potentials in Lebesgue spaces have been considered, while our result covers more singular potentials in Lorentz spaces. 
\end{itemize}
\end{remark}

The rest of this paper devoted to the proof of the above theorems and is organized as follows. In the next section we collect several basic results used in later sections. In particular, we recall there uniform Sobolev inequalities for the free resolvent. Then we prove uniform Sobolev inequalities for $(H_0-z)^{-1}$ and $(H_\delta-z)^{-1}$ in Section \ref{section_Sobolev}. The proof of Theorems \ref{theorem_1} and \ref{theorem_2} is given in Section \ref{section_proof}.  In appendix \ref{appendix_form_compact}, we recall very briefly a basic concept of $m$-sectorial operators, which particularly gives the precise definition of $H_0$ and $H_\delta$. In Appendix \ref{appendix_interpolation}, a brief review on the real interpolation and Lorentz spaces is given. We prove the relative form compactness of $V$ in Appendix \ref{appendix_remark}. In Appendix \ref{appendix_resolvent}, the proof of weighted resolvent estimates in Section \ref{section_Sobolev} is given.

\section{Preliminaries}
\label{section_preliminaries}

In this section we record several basic results which will be used to prove uniform Sobolev inequalities for our operators $H_0$ and $H_\delta$ in Section \ref{section_Sobolev}. 

\subsection{Abstract resolvent estimates}
We first prepare an abstract method which enables us to deduce the proof of resolvent estimates between Banach spaces to that of weighted estimates in a Hilbert space. 
Let $(T_0,D(T_0))$, $(T,D(T))$ be self-adjoint operators on a complex separable Hilbert space $\H$ equipped with norm $\norm{\cdot}$ and inner product $\<\cdot,\cdot\>$. Suppose there exist densely defined closed operators $(Y,D(Y))$ and $(Z,D(Z))$ such that $T=T_0+Y^*Z$ in the sense that 
\begin{align}
\label{lemma_abstract_1_-1}
D(T_0) &\cup D(T)\subset D(Y) \cap D(Z),\\
\label{lemma_abstract_1_0}
\<T\varphi,\psi\>&=\<\varphi,T_0\psi\>+\<Z\varphi,Y\psi\>,\quad \varphi\in D(T),\ \psi\in D(T_0).
\end{align}
Note that $D(Z^*)$ is dense in $\H$ since $Z$ is closed (see \cite[Theorem VIII.1]{ReSi}). 

Recall that a pair of two Banach spaces $(\A_1,\A_2)$ is said to be a Banach couple if both $\A_1$ and $\A_2$ can be algebraically and topologically embedded in a Hausdorff topological vector space $\A$. Given a Banach couple $(\A_1,\A_2)$, $\A_1\cap \A_2$ is well-defined. 
\begin{lemma}
\label{lemma_abstract_1}
Let $\A,\B$ be two Banach spaces such that $(\mathcal H,\mathcal A)$ and $(\mathcal H,\mathcal B)$ are Banach couples. 
Let $W$ be a bounded self-adjoint operator on $\H$ with $\norm{W}_{\H\to \H}\le1$ such that $W$ commutes with $T_0$ and $Y$ and that $WD(Z^*)\subset D(Z^*)$. Let $z\in \rho(T_0)\cap \rho(T)$. Suppose there exist constants $r_1,..,r_5$ (possibly depending on $z$) such that
\begin{align}
\label{lemma_abstract_1_1}
|\<W(T_0-z)^{-1}W\varphi,\psi\>|&\le r_1\norm{\varphi}_{\A}\norm{\psi}_{\B},\\
\label{lemma_abstract_1_4}
\norm{Z(T_0-z)^{-1}\varphi} &\le r_2\norm{\varphi}_{\A},\\
\label{lemma_abstract_1_2}
\norm{Y(T_0-z)^{-1}\varphi} &\le r_3\norm{\varphi}_{\A},\\
\label{lemma_abstract_1_3}
\norm{Y(T_0-\overline z)^{-1}\psi} &\le r_4\norm{\psi}_{\B},\\
\label{lemma_abstract_1_5}
\norm{WZ(T-\overline z)^{-1}Z^*Wh} &\le r_5\norm{h} ,
\end{align}
for all $\varphi\in \H\cap \A$, $\psi\in \H\cap \B$ and $h\in D(Z^*)$. Then, for all $\varphi\in \H\cap \A$ and $\psi\in \H\cap \B$, 
\begin{align}
\label{lemma_abstract_1_6}
|\<W(T-z)^{-1}W\varphi,\psi\>|\le (r_1+r_2r_4+r_3r_4r_5)\norm{\varphi}_{\A}\norm{\psi}_{\B}.
\end{align}
\end{lemma}
Note that, by \eqref{lemma_abstract_1_-1}, both $(T_0-z)^{-1}$ and $(T-z)^{-1}$ map from $\H$ to $D(Y)\cap D(Z)$ if $z\in \rho(T_0)\cap \rho(T)$. In particular, the left hand sides of \eqref{lemma_abstract_1_1}--\eqref{lemma_abstract_1_6} are well defined. 

\begin{proof}
Although the proof is essentially same as that of \cite[Proposition 4.1]{BoMi} (in which the case $W\equiv1$ was considered), we give its details for the same of completeness. 

Let $R_{T_0}(z)=(T_0-z)^{-1}$ and $
R_{T}(z)=(T-z)^{-1}$.  At a formal level, \eqref{lemma_abstract_1_6} follows from \eqref{lemma_abstract_1_1}--\eqref{lemma_abstract_1_5} and an iterated resolvent identity:
$$
R_{T}(z)=R_{T_0}(z)-R_{T_0}(z)Y^*ZR_{T_0}(z)+R_{T_0}(z)Y^*ZR_T(z)Z^*YR_{T_0}(z).
$$
We however note that this formula is slightly formal since neither $Y$ nor $Z$ is bounded on $\H$. Instead, the proof is based on the following resolvent identities in a weak sense:
\begin{align}
\label{proof_lemma_abstract_1_1}
\<R_{T}(z)\varphi,\psi\>
&=\<R_{T_0}(z)\varphi,\psi\>-\<ZR_{T}(z)\varphi,YR_{T_0}(\overline z)\psi\>,\\
\label{proof_lemma_abstract_1_2}
&=\<R_{T_0}(z)\varphi,\psi\>-\<YR_{T_0}(z)\varphi,ZR_{T}(\overline z)\psi\>
\end{align}
for $\varphi,\psi\in \H$. \eqref{proof_lemma_abstract_1_1} follows from \eqref{lemma_abstract_1_0} with $f=R_{T}(z)\varphi$, $g=R_{T_0}(\overline z)\psi$, while \eqref{proof_lemma_abstract_1_2} is derived by exchanging the roles of $\varphi$ and $\psi$
in \eqref{proof_lemma_abstract_1_1} and replacing $z$ by $\overline z$. 

Let $\varphi\in \H\cap \A$, $\psi\in \H\cap \B$. Then \eqref{lemma_abstract_1_1},  \eqref{lemma_abstract_1_2} and the condition $\norm{W}_{\H\to \H}\le 1$ imply
\begin{equation}
\begin{aligned}
|\<WR_{T}(z)W\varphi,\psi\>|
&\le |\<WR_{T_0}(z)W\varphi,\psi\>|+|\<WZR_{T}(z)W\varphi,YR_{T_0}(\overline z)\psi\>|\\
\label{proof_lemma_abstract_1_3}
&\le r_1\norm{\varphi}_\A\norm{\psi}_\B+r_4\norm{WZR_{T}(z)W\varphi} \norm{\psi}_\B,
\end{aligned}
\end{equation}
Taking into account the fact that $D(Z^*)$ is dense in $\H$, it remains to deal with 
$$
\norm{WZR_{T}(z)W\varphi} =\sup\limits_{h\in D(Z^*),\, \norm{h} =1}|\<WZR_{T}(z)W\varphi,h\>|.
$$
By \eqref{lemma_abstract_1_4}, \eqref{lemma_abstract_1_2}, \eqref{lemma_abstract_1_5} and \eqref{proof_lemma_abstract_1_2} , we have
\begin{align*}
|\<WZR_{T}(z)W\varphi,h\>|
&\le |\<R_{T_0}(z)W\varphi,Z^*Wh\>|+|\<YR_{T_0}W\varphi,ZR_{T}(\overline z)Z^*Wh\>|\\
&\le \norm{ZR_{T_0}(z)W\varphi} 
+\norm{YR_{T_0}\varphi} \norm{WZR_{T}(\overline z)Z^*Wh} ,\\
&\le (r_2+r_3r_5)\norm{\varphi}_{\A}. 
\end{align*}
This bound, together with \eqref{proof_lemma_abstract_1_3}, implies \eqref{lemma_abstract_1_6}. 
\end{proof}

\begin{remark}
It is also possible to state a similar criterion to obtain the boundedness of $W(T-z)^{-1}W$ in $\mathbb B(\A,\B^*)$ from the corresponding boundedness of $W(T_0-z)^{-1}W$ and \eqref{lemma_abstract_1_4}--\eqref{lemma_abstract_1_5}. 
However, it requires additional assumptions on $\A$, $\B$ and their dual spaces, which makes the proof rather involved. On the other hand, in concrete applications, boundedness of $W(T-z)^{-1}W$ can be easily seen from \eqref{lemma_abstract_1_6} and a standard duality argument. 
\end{remark}

\begin{remark}
In Section \ref{section_Sobolev}, we will apply this lemma to the case when $\H=L^2(\R^n)$, $\A$ and $\B$ are certain Lorentz spaces, $T_0=-\Delta$, $T=H_0$ or $H_\delta$, $Z,Y\in\{|x|^{-1},|x|V_\delta\}$ and $W$ is the  orthogonal projection onto non-radially symmetric functions. 
\end{remark}

\subsection{Uniform Sobolev inequalities for the free resolvent}
In this subsection we recall the uniform Sobolev inequality for the free resolvent $(-\Delta-z)^{-1}$, proved by Kenig-Ruiz-Sogge \cite{KRS}, Guti\'errez \cite{Gut} and Frank-Simon \cite{FrSi}, which plays a important role in the next section. We first introduce several notation. In what follows $p'=p/(p-1)$ denotes the usual H\"older conjugate of $p$. 
Define 10 points in $(\frac12,1)\times(0,\frac12)$ by 
\begin{align*}
A&=\Big(\frac{1}{p_A},\frac{1}{q_A}\Big):=\Big(\frac{n+1}{2n},\frac{n-3}{2n}\Big)
,\quad
A'=\Big(\frac{1}{q_A'},\frac{1}{p_A'}\Big)
,\\
B&:=\Big(\frac{1}{p_A},\frac{1}{p_A}-\frac{2}{n+1}\Big)
,\quad
B'=\Big(\frac{1}{p'_A}+\frac{2}{n+1},\frac{1}{p_A'}\Big)
,\\
C&:=\Big(\frac{n+2}{2n},\frac{n-2}{2n}\Big),\\
D&:=\Big(\frac{n+3}{2(n+1)},\frac{n-1}{2(n+1)}\Big),\\
E&:=\Big(\frac{1}{p_A},\frac{1}{p_A'}\Big)=\Big(\frac{n+1}{2n},\frac{n-1}{2n}\Big),\\
F&:=
\Big(\frac{n^2+3n-2}{2n(n+1)},\frac{n-2}{2n}\Big),\quad
F'=
\Big(\frac{n+2}{2n},\frac{n^2-n+2}{2n(n+1)}\Big),\\
G&:=
\Big(\frac{n^2+3n-2}{2n(n+1)},\frac{n^2-n+2}{2n(n+1)}\Big),
\end{align*}
(see Figure \ref{figure_1} below). Let $\Omega_0\subset (\frac12,1)\times(0,\frac12)$ be a solid region defined by a trapezium $ABB'A'$ with two closed line segments $AB$ and $B'A'$ removed, and $\Omega\subset \Omega_0$ be a solid region defined by an isosceles right triangle $\triangle FCF'$ (with two points $F,F'$ removed if $n=3$), namely 
\begin{align*}
\Omega_0&:=\left\{\left(\frac1p,\frac1q\right)\ \Big|\ \frac{2}{n+1}\le \frac{1}{p}-\frac{1}{q}\le \frac{2}{n},\ \frac{n+1}{2n}<\frac 1p,\ \frac1q<\frac{n-1}{2n}\right\},\\
\Omega&:=\left\{\left(\frac1p,\frac1q\right)\in \Omega_0\ \Big|\ \frac 1p\le \frac{n+2}{2n},\ \frac{n-2}{2n}\le \frac 1q\right\}.
\end{align*}
Note that $(1/p,1/q)\in \overline{AA'}$ (resp. $(1/p,1/q)\in \overline{BB'}$) satisfies 
$$
\frac1p-\frac1q=\frac 2n\quad \Big(\text{resp.}\ \frac1p-\frac1q=\frac{2}{n+1}\Big)
$$
and that $F=B$, $F'=B'$ and $E=G$ if $n=3$. 

For $1<p<\infty$ and $1\le s,q\le \infty$, we define a vector-valued Lorentz space $L^{p,s}_rL^q_\omega(\R^n):=L^{p,s}(\R_+,r^{n-1}dr;L^q(\mathbb S^{n-1}))$ through the norm
$$
\norm{f}_{L^{p,s}_rL^q_\omega(\R^n)}:=\big|\big|\norm{f(r\cdot)}_{L^q(\mathbb S^{n-1})}\big|\big|_{L^{p,s}((0,\infty),r^{n-1}dr)}.
$$
It is easy to see that $L^{p}_rL^q_\omega(\R^n)\subset L^{p}(\R^n)$ and $\norm{f}_{L^{p}}\le |\mathbb S|^{1/p-1/q}\norm{f}_{L^{p}_rL^q_\omega}$ if $p\le q$. By the real interpolation (see Theorems \ref{theorem_interpolation_2} and \ref{theorem_Lorentz_2}), we also have $L^{p,s}_rL^q_\omega(\R^n)\subset L^{p,s}(\R^n)$ and
\begin{align}
\label{mixed}
\norm{f}_{L^{p,s}(\R^n)}\le |\mathbb S|^{1/p-1/q}\norm{f}_{L^{p,s}_rL^q_\omega(\R^n)}
\end{align}
for all $1<p\le q<\infty$ and $1\le s\le\infty$. Furthermore, if $f$ is radially symmetric, then $\norm{f}_{L^{p,s}(\R^n)}$ is equivalent to $\norm{f}_{L^{p,s}((0,\infty),r^{n-1}dr)}$ and $\norm{f}_{L^{p,s}_rL^q_\omega(\R^n)}$ for any $1\le q\le\infty$.

\begin{proposition}[Free uniform Sobolev inequalities]
\label{proposition_free_1}$ $\\
{\rm (1)} Let $n\ge3$ and $(1/p,1/q)\in \Omega_0$. Then there exists $C_{n,p,q}>0$ such that 
\begin{align}
\label{proposition_free_1_1}
\norm{(-\Delta-z)^{-1}}_{L^{p,s}(\R^n)\to L^{q,s}(\R^n)}\le C_{n,p,q} |z|^{\frac n2\left(\frac1p-\frac1q-\frac 2n\right)}
\end{align}
 for all $1\le s\le\infty$ and $z\in \C\setminus[0,\infty)$. \\
{\rm (2)} Let $n\ge2$ and $(1/p,1/p')\in\overline{CE}\setminus\{C,E\}$. 
Then there exists $C_{n,p}>0$ such that
\begin{align}
\label{proposition_free_1_2}
\norm{(-\Delta-z)^{-1}}_{L^{p,s}_rL^2_\omega(\R^n)\to L^{p',s}_rL^2_\omega(\R^n)}\le C_{n,p} |z|^{-\frac{n+2}{2}+\frac np}
\end{align}
 for all $1\le s\le\infty$ and $z\in \C\setminus[0,\infty)$
\end{proposition}

\begin{proof}
Assuming $|z|=1$ and $z\neq1$, we first consider estimates in Lebesgue spaces, \emph{i.e.}, estimates with all of $L^{t,s}$ and $L^{t,s}_rL^2_\omega$ replaced by $L^t$ and $L^{t}_rL^2_\omega$, respectively. Then \eqref{proposition_free_1_1} was proved by \cite[Theorems 2.2 and 2.3]{KRS} for $(1/p,1/q)\in (\overline{AA'}\cup\overline{CD})\setminus\{A,A'\}$ and extended by \cite[Theorem 6]{Gut} to any $(1/p,1/q)\in \Omega_0$. \eqref{proposition_free_1_2} 
is due to \cite[Theorem 4.3]{FrSi}. Next, estimates in Lorentz spaces follow from real interpolation, more precisely, from Theorems \ref{theorem_interpolation_2} and \ref{theorem_Lorentz_2} in Appendix \ref{appendix_interpolation}. Finally, in order to remove the restriction $|z|\ge1$, it suffices to apply \eqref{proposition_free_1_1} and \eqref{proposition_free_1_2} to $f(|z|^{1/2} x)$ and use the scaling properties $\norm{f(\lambda x)}_{L^{p,s}(\R^n)}=\lambda^{-n/p}\norm{f}_{L^{p,s}(\R^n)}$ and $\norm{f(\lambda x)}_{L^{p,s}_rL^2_\omega(\R^n)}=\lambda^{-n/p}\norm{f}_{L^{p,s}_rL^2_\omega(\R^n)}$ for $\lambda>0$. 
\end{proof}

\begin{figure}[htbp]
\begin{center}
\includegraphics[width=90mm]{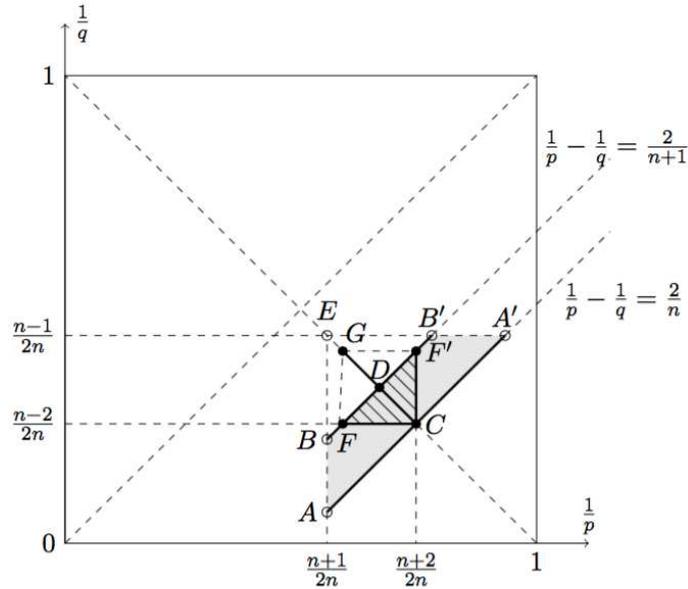}
\end{center}
\vspace{-1.5cm}
\caption{The admissible region $\Omega_0$ for Proposition \ref{proposition_free_1} (1) is the shaded region defined by the trapezium $ABB'A'$ with two closed line segments $\overline{AB}$ and $\overline{B'A'}$ removed. The admissible region $\Omega\subset \Omega_0$ for Theorem \ref{theorem_Sobolev_2} (1) is the oblique lined region defined by the isosceles right triangle $\triangle FCF'$ (with two points $F,F'$ removed if $n=3$). }
\label{figure_1}
\end{figure}
\section{Uniform Sobolev inequalities}
\label{section_Sobolev}

The purpose of this section is to prove uniform Sobolev inequalities for $(H_0-z)^{-1}$ and $(H_\delta-z)^{-1}$, which are main ingredients for the proof of the main theorems. 

In what follows we assume $n\ge3$. Let $P$ be the projection onto the space of radial functions
$$
Pf(x)=\frac{1}{|\mathbb S^{n-1}|}\int_{\mathbb S^{n-1}}f(|x|\omega)d\sigma(\omega)
$$
and set $P^\perp=\Id-P$. 
We begin with the following proposition which plays a role in this section. 

\begin{proposition}[Weighted resolvent estimates]
\label{proposition_Sobolev_1}$ $\\
{\rm(1)} The critical case: Let $H_0=-\Delta-C_{\Hardy}|x|^{-2}$. Then
\begin{align}
\label{proposition_Sobolev_1_1}
\sup_{z\in \C\setminus[0,\infty)}\norm{|x|^{-1}P^\perp (H_0-z)^{-1}P^\perp |x|^{-1}}_{L^2(\R^n)\to L^2(\R^n)}<\infty.
\end{align}
{\rm(2)} The subcritical case: Let $H_\delta=-\Delta+V_\delta$ with $\delta>0$. Then
\begin{align}
\label{proposition_Sobolev_1_2}
\sup_{z\in \C\setminus[0,\infty)}\norm{|x|^{-1} (H_\delta-z)^{-1} |x|^{-1}}_{L^2(\R^n)\to L^2(\R^n)}\le C\delta^{-2}.
\end{align}
Here $C>0$ may be taken uniformly with respect to $\delta$. 
\end{proposition}

\begin{proof}
\eqref{proposition_Sobolev_1_2} was proved by \cite[Section 2]{BVZ} (see also \cite[Appendix B]{BoMi}) and the proof of \eqref{proposition_Sobolev_1_1} is essentially same. For the sake of completeness, we give its details in Appendix \ref{appendix_resolvent}. 
\end{proof}

Now we state the main result in the critical case: 
\begin{theorem}	[Uniform Sobolev inequalities in the critical case]
\label{theorem_Sobolev_2}$ $\\
{\rm (1)} Let $(1/p,1/p')\in \overline{CD}\setminus\{C\}$. 
Then
\begin{align}
\label{theorem_Sobolev_2_1}
\norm{(H_0-z)^{-1}}_{L^{p,2}(\R^n)\to L^{p',2}(\R^n)}\le C_{n,p}|z|^{-\frac{n+2}{2}+\frac np}
,\ z\in \C\setminus[0,\infty). 
\end{align}
{\rm (2)} Let $(1/p,1/p')\in \overline{CG}\setminus\{C\}$ if $n\ge4$ or $(1/p,1/p')\in \overline{CG}\setminus\{C,G\}$ if $n=3$. 
Then
\begin{align}
\label{theorem_Sobolev_2_2}
\norm{(H_0-z)^{-1}}_{L^{p,2}_rL^2_\omega(\R^n)\to L^{p',2}_rL^2_\omega(\R^n)}
\le C_{n,p}|z|^{-\frac{n+2}{2}+\frac np}
,\ z\in \C\setminus[0,\infty). 
\end{align}
\end{theorem}
Note that both \eqref{theorem_Sobolev_2_1} and \eqref{theorem_Sobolev_2_2} can fail for $(1/p,1/p')\in C$, {\it i.e.}, $p=\frac{2n}{n+2}$ (see Remark \ref{remark_2}). Before proving this theorem, we state a corollary  which will be used to prove Corollary \ref{corollary_3} in the next section. Define two functions $K_\gamma,K_{\gamma,\ep}^{\mathrm{rad}}:\C\setminus[0,\infty)\to \R_+$ by
\begin{align*}
K_\gamma(z)
&:=|z|^{\min(\gamma,1/2)}d(z)^{(\gamma-1/2)_+},\\
K_{\gamma,\ep}^{\mathrm{rad}}(z)
&:=|z|^{\min(\gamma,n/(n-1)-\ep)}d(z)^{(\gamma-n/(n-1)+\ep)_+}. 
\end{align*}
\begin{corollary}
\label{corollary_Sobolev_2}
Let $\gamma>0$, $1/p_\gamma=1/(n+2\gamma)+1/2$, $\ep=0$ if $n\ge4$ and $\ep>0$ if $n=3$. Then
\begin{align}
\label{corollary_Sobolev_2_1}
\norm{(H_0-z)^{-1}}_{L^{p_\gamma,2}(\R^n)\to L^{p_\gamma',2}(\R^n)}
&\le C_{\gamma,n} 
K_\gamma(z)^{-\frac{1}{n/2+\gamma}},\\
\label{corollary_Sobolev_2_2}
\norm{(H_0-z)^{-1}}_{L^{p_\gamma,2}_rL^2_\omega(\R^n)\to L^{p_\gamma',2}_rL^2_\omega(\R^n)}
&\le C_{\gamma,n,\ep}
K_{\gamma,\ep}^{\mathrm{rad}}(z)^{-\frac{1}{n/2+\gamma}}
\end{align}
for all $z\in \C\setminus[0,\infty)$ with some $C_{\gamma,n},C_{\gamma,n,\ep}>0$ being independent of $z$. 
\end{corollary}

\begin{proof}
Observe that $0<\gamma\le 1/2$ if and only if 
\begin{align}
\label{proof_corollary_Sobolev_2_00}
\frac{n+3}{2(n+1)}\le \frac{1}{p_\gamma}<\frac{n+2}{2n},\ \text{\emph{i.e.},}\ \Big(\frac{1}{p_\gamma},\frac{1}{p'_\gamma}\Big)\in \overline{CD}\setminus\{C\}.
\end{align}
Similarly, $0<\gamma\le n/(n-1)$ if and only if
\begin{align}
\label{proof_corollary_Sobolev_2_01}
\frac{n^2+3n-2}{2n(n+1)}\le \frac{1}{p_\gamma}<\frac{n+2}{2n},\ \text{\emph{i.e.},}\ \Big(\frac{1}{p_\gamma},\frac{1}{p'_\gamma}\Big)\in \overline{CG}\setminus\{C\}.
\end{align}
Let us first show \eqref{corollary_Sobolev_2_1}. Having \eqref{proof_corollary_Sobolev_2_00} in mind, the result for $0<\gamma\le1/2$ follows from \eqref{theorem_Sobolev_2_1} since
$$
|z|^{\frac{n+2}{2}-\frac{n}{p_\gamma}}=|z|^{-\frac{\gamma}{n/2+\gamma}}=K_\gamma(z)^{-\frac{1}{n/2+\gamma}}. 
$$
 Let $\gamma>1/2$ and $\theta\in(0,1)$ be such that
$$
1-\theta=\frac{1/p_\gamma-1/2}{1/p_{1/2}-1/2}=\frac{n+1}{n+2\gamma}
\quad\text{\emph{i.e.},}\quad\theta=\frac{2\gamma-1}{n+2\gamma}, 
$$
where $p_{1/2}=2(n+1)/(n+3)$. Interpolating between \eqref{theorem_Sobolev_2_1} with $p=p_{1/2}$ and the trivial bound 
\begin{align}
\label{proof_corollary_Sobolev_2_1}
\norm{(H_0-z)^{-1}}_{L^2\to L^2}=d(z)^{-1},
\end{align}
by means of Theorems \ref{theorem_interpolation_2} and \ref{theorem_Lorentz_2}, we have
\begin{align*}
\norm{(H_0-z)^{-1}}_{L^{p_\gamma',2}\to L^{p_\gamma,2}}
&\le C_{n,p_{1/2}}^{1-\theta}
|z|^{-\frac{1-\theta}{n+1}}
d(z)^{-\theta}=(C_{n,p_{1/2}})^{\frac{n+1}{n+2\gamma}}K_\gamma(z)^{-\frac{1}{n/2+\gamma}}.
\end{align*}

Next, by virtue of the equivalence \eqref{proof_corollary_Sobolev_2_01}, \eqref{corollary_Sobolev_2_2} for $0<\gamma\le n/(n-1)$ if $n\ge4$ or for $0<\gamma<3/2$ if  $n=3$ is nothing but \eqref{theorem_Sobolev_2_2}. When $\gamma>n/(n-1)$ and $n\ge4$, we take $\theta\in(0,1)$ such that
$$
1-\theta=\frac{1/p_\gamma-1/2}{1/p_{n/(n-1)}-1/2}=\frac{n(n+1)}{(n+2\gamma)(n-1)}\quad
\text{\emph{i.e.},}\quad\theta=\frac{2\gamma-n/(n-1)}{n+2\gamma},
$$
where $p_{n/(n-1)}=n(n+1)/(n^2+3n-2)$. 
Interpolating between \eqref{theorem_Sobolev_2_2} with $p=p_{n/(n-1)}$ and \eqref{proof_corollary_Sobolev_2_1}, we then have \eqref{corollary_Sobolev_2_2} since 
\begin{align*}
(1-\theta)\Big(-\frac{n+2}{2}+\frac{n}{p_{n/n-1}}\Big)
=-\frac{n}{(n/2+\gamma)(n-1)}. 
\end{align*}
\eqref{corollary_Sobolev_2_2} for $n=3$ and $\gamma\ge 3/2$ follows by interpolating between \eqref{theorem_Sobolev_2_2} with $p=p_{3/2-\ep}$ and \eqref{proof_corollary_Sobolev_2_1}. 
\end{proof}

Next we prove Theorem \ref{theorem_Sobolev_2}, where the proof consists of two parts. We first  observe that, since $H_0$ is rotationally invariant, both $P$ and $P^\perp$ commute with $H_0$ and thus with $(H_0-z)^{-1}$. Hence $(H_0-z)^{-1}$ can be decomposed as
$$
(H_0-z)^{-1}=P(H_0-z)^{-1}P+P^\perp (H_0-z)^{-1}P^\perp.
$$ 
We begin with the non-radial part $P^\perp (H_0-z)^{-1}P^\perp$ which satisfies uniform Sobolev inequalities even for $p=2n/(n+2)$ as follows:
\begin{proposition}
\label{proposition_Sobolev_4}
Let $(1/p,1/q)\in \Omega$ and $z\in \C\setminus[0,\infty)$. Then
\begin{align}
\label{proposition_Sobolev_4_1}
\norm{P^\perp (H_0-z)^{-1}P^\perp}_{L^{q,2}(\R^n)\to L^{p,2}(\R^n)}
\le C_{n,p,q}|z|^{\frac n2\left(\frac1p-\frac1q-\frac 2n\right)}. 
\end{align}
Furthermore, if $(1/p,1/p')\in \overline{CG}$ for $n\ge4$ or $(1/p,1/p')\in \overline{CG}\setminus\{G\}$ for $n=3$, then
\begin{align}
\label{proposition_Sobolev_4_2}
\norm{P^\perp(H_0-z)^{-1}P^\perp}_{L^{p',2}_rL^2_\omega(\R^n)\to L^{p,2}_rL^2_\omega(\R^n)}
\le C_{n,p}|z|^{-\frac{n+2}{2}+\frac np}.
\end{align}
\end{proposition}

\begin{proof}
The proof is based on Lemma \ref{lemma_abstract_1} and Proposition \ref{proposition_free_1}. 
We first let $(1/p,1/q)\in \Omega$ and prove \eqref{proposition_Sobolev_4_1} by means of Lemma \ref{lemma_abstract_1} with $\H=L^2$, $\A=L^{p,2}$, $\B=L^{q',2}$, $T_0=-\Delta$, $T=H_0$, $Y=-C_{\Hardy}|x|^{-1}$, $Z=|x|^{-1}$ and $W=P^\perp$. Since $(1/p,1/q)$, $(1/p,(n-2)/(2n))$, $(1/q',(n-2)/(2n))\in \Omega_0$ (see Figure \ref{figure_1}), Proposition \ref{proposition_free_1} and H\"older's inequality yield that for all $\varphi\in L^2\cap L^{p,2}$, $\psi\in L^2\cap L^{q',2}$, 
\begin{align}
\label{proof_Sobolev_4_2}
\norm{|x|^{-1}(-\Delta-z)^{-1}\varphi}_{L^{2}}
&\le C\norm{(-\Delta-z)^{-1}\varphi}_{L^{\frac{2n}{n-2},2}}
\le C|z|^{\frac n2\left(\frac1p-\frac{n-2}{2n}-\frac 2n\right)}\norm{\varphi}_{L^{p,2}},\\
\nonumber
\norm{|x|^{-1}(-\Delta-z)^{-1}\psi}_{L^{2}}
&\le C|z|^{\frac n2\left(1-\frac1q-\frac{n-2}{2n}-\frac 2n\right)}\norm{\psi}_{L^{q',2}}. 
\end{align}
Using \eqref{proposition_free_1_1},  \eqref{proposition_Sobolev_1_1} and these two estimates, we can apply Lemma \ref{lemma_abstract_1} to obtain
$$
|\<P^\perp(H_0-z)^{-1}P^\perp\varphi,\psi\>|\le C|z|^{\frac n2\left(\frac1p-\frac1q-\frac 2n\right)}\norm{\varphi}_{L^{p,2}}\norm{\psi}_{L^{q',2}},
$$
which implies \eqref{proposition_Sobolev_4_1} by density and duality arguments. 

In order to show \eqref{proposition_Sobolev_4_2} for $(1/p,1/p')\in \overline{CG}\setminus\{G\}$, we observe that $({1}/{p},{(n-2)}/{(2n)})\in \Omega_0$ (see Figure \ref{figure_1}). Taking the fact $L^{p,2}_rL^2_\omega\subset L^{p,2}$ and $L^{p',2} \subset L^{p',2}_rL^2_\omega$ since $p<2$ into account and using \eqref{proposition_free_1_2}, \eqref{proposition_Sobolev_1_1} and \eqref{proof_Sobolev_4_2}, one can use Lemma \ref{lemma_abstract_1} with $\A=\B=L^{p,2}_rL^2_\omega$ to obtain \eqref{proposition_Sobolev_4_2}  for $(1/p,1/p')\in \overline{CG}\setminus\{G\}$ and $n\ge3$. Finally, when $n\ge4$, $F\in \Omega_0$ so we thus  have the endpoint estimate \eqref{proposition_Sobolev_4_2} with $(1/p,1/p')=G$ by the same argument. 
\end{proof}

It remains to deal with the radially symmetric part $P(H_0-z)^{-1}P$. 
\begin{proposition}
\label{proposition_Sobolev_5}
For any $(1/p,1/p')\in \overline{CE}\setminus\{C,E\}$ there exists $C_{n,p}>0$ such that
$$
\norm{P (H_0-z)^{-1}P}_{L^{p,2}(\R^n)\to L^{p',2}(\R^n)}\le C_{n,p}|z|^{\frac{n+2}{2}-\frac np},\quad z\in \C\setminus[0,\infty).
$$
\end{proposition}

\begin{proof}
Since $\overline{CE}\setminus\{C,E\}$ is an open line segment, by virtue of real interpolation (Theorems \ref{theorem_interpolation_2} and \ref{theorem_Lorentz_2}), it suffices to show corresponding $L^p$-$L^{p'}$ estimates. As above we may assume $|z|=1$. 

The result is deduced from the two-dimensional free uniform Sobolev inequalities as follows. Recall that $H_0$ restricted to radial functions is unitarily equivalent to $-\Delta_{\R^2}$ restricted to radial functions by means of the unitary map $L^2_{\mathrm{rad}}(\R^n)\ni f\mapsto \sqrt{\frac{|\mathbb S^{n-1}|}{2\pi}}|x|^{(n-2)/2}f\in L^2_{\mathrm{rad}}(\R^2)$, namely $$H_0P=-|x|^{-(n-2)/2}\Delta_{\R^2}|x|^{(n-2)/2}P$$
on $D(H_0)$. This formula follows from the identity 
$$
\<-|x|^{-(n-2)/2}\Delta_{\R^2}|x|^{(n-2)/2}f,g\>=\left\langle \Big(-\frac{d^2}{dr^2}-\frac{n-1}{r}\frac{d}{dr}-\frac{(n-2)^2}{4r^2}\Big)f,g\right\rangle=\<H_0f,g\>
$$ 
for all radially symmetric functions $f,g\in C_0^\infty(\R^n)$ and the fact that $C_0^\infty(\R^n)$ is dense in $D(q_0)$, where $r=|x|$.  
In particular, we have
\begin{align}
\label{proof_proposition_Sobolev_5_1}
P (H_0-z)^{-1}P=P|x|^{-(n-2)/2}(-\Delta_{\R^2}-z)^{-1}|x|^{(n-2)/2}P. 
\end{align}
Let $f\in C_0^\infty(\R^n)$ and  set $v=(-\Delta_{\R^2}-z)^{-1}|x|^{(n-2)/2}Pf$. 
Take $q,\tau_1\ge1$ so that
$$
\frac{n+2}{2}-\frac np=2-\frac{2}{q},\quad \frac{1}{p'}=\frac{1}{\tau_1}+\frac{1}{q'},\quad 
\frac{1}{\tau_1}=\frac{(n-2)}{2}\Big(\frac12-\frac{1}{p'}\Big). 
$$
Note that $1<q<4/3$ and hence $(1/q,1/q')$ satisfies the admissible condition for free uniform Sobolev inequalities \eqref{proposition_free_1_2} with $n=2$. Then we learn by \eqref{proof_proposition_Sobolev_5_1} and H\"older's inequality that
\begin{align*}
\norm{P (H_0-z)^{-1}P f}_{L^{p'}(\R^n)}
&=\Big(\frac{|\mathbb S^{n-1}|}{2\pi}\Big)^{1/p'}\big|\big||x|^{-(n-2)\left(\frac12-\frac{1}{p'}\right)}v\big|\big|_{L^{p'}(\R^2)}\\
&\le C\big|\big|{|x|^{-\frac2\tau_1}}\big|\big|_{L^{\tau_1,\infty}(\R^2)}\norm{v}_{L^{q',p'}(\R^2)}\\
&\le C\norm{v}_{L^{q',p'}(\R^2)}.
\end{align*}
Taking into account the fact that both $v$ and $|x|^{(n-2)/2}Pf$ are radially symmetric,  
one can use \eqref{proposition_free_1_2} with $n=2$ to obtain
\begin{align*}
\norm{v}_{L^{q',p'}(\R^2)}
&\le C|z|^{2-\frac2q}\big|\big|{|x|^{\frac{n-2}{2}}P f}\big|\big|_{L^{q,p'}(\R^2)}\\
&=C|z|^{\frac{n+2}{2}-\frac np}\big|\big|{|x|^{\frac{n-2}{2}}P f}\big|\big|_{L^{q,p'}(\R^2)}.
\end{align*}
Since $L^{q,p}\,\hookr\,L^{q,p'}$ (note that $p'>p$), H\"older's inequality with exponents
$$
\frac{1}{q}=\frac{1}{\tau_2}+\frac{1}{p},\quad \frac{1}{\tau_2}=\frac{(n-2)}{2}\Big(\frac12-\frac1p\Big)
$$
yields that
\begin{align*}
\norm{|x|^{\frac{n-2}{2}}P f}_{L^{q,p'}(\R^2)}
&\le C\big|\big|{|x|^{\frac{n-2}{2}}Pf}\big|\big|_{L^{q,p}(\R^2)}\\
&\le C\big|\big|{|x|^{-\frac{2}{\tau_2}}}\big|\big|_{L^{\tau_2,\infty}(\R^2)}\big|\big|{|x|^{\frac{n-2}{p}}Pf}\big|\big|_{L^{p}(\R^2)}\\
&\le C\norm{f}_{L^p(\R^n)},
\end{align*}
which completes the proof. 
\end{proof}

\begin{proof}[Proof of Theorem \ref{theorem_Sobolev_2}]
Theorem \ref{theorem_Sobolev_2} readily follows from Propositions \ref{proposition_Sobolev_4} and \ref{proposition_Sobolev_5}. 
\end{proof}

\begin{remark}
\label{remark_2}
Let $\kappa=-iz^{1/2}$ with $z=-\kappa^2$ for $\Re z<0$. It is known (see, \emph{e.g}, \cite{JeNe}) that$$(-\Delta_{\R^2}+\kappa^2)^{-1}=-\frac{\log \kappa}{2\pi} P_0+O(1),\quad |\kappa|\to0$$ in $\mathbb B(L^2(\R^2,\<x\>^{s}dx),L^2(\R^2,\<x\>^{-s}dx))$ for $s>3/2$, where $P_0f:=\int_{\R^2} f(x)dx$. Together with \eqref{proof_proposition_Sobolev_5_1} and \eqref{proposition_Sobolev_4_1} for $(p,{p'})=(\frac{2n}{n+2},\frac{2n}{n-2})$, this implies the low energy asymptotics:
$$(H_0+\kappa^2)^{-1}=-\frac{\log \kappa}{2\pi} P|x|^{-\frac{n-2}{2}}P_0|x|^{\frac{n-2}{2}}P+O(1),\quad |\kappa|\to0$$in $\mathbb B(L^2(\R^n,\<x\>^{s}dx),L^2(\R^n,\<x\>^{-s}dx))$ for sufficiently large $s$. In particular, $(H_0-z)^{-1}$ has a logarithmic singularity at $z=0$ and both \eqref{theorem_Sobolev_2_1} and \eqref{theorem_Sobolev_2_2}  thus cannot hold for $p=\frac{2n}{n+2}$. \end{remark}

Next we consider the subcritical case. The main result in this case is the following. 
\begin{theorem}	[Uniform Sobolev inequalities in the subcritical case]
\label{theorem_Sobolev_6}
Let $n\ge3$ and $\delta>0$. Then, for any $(1/p,1/q)\in\Omega$ there exists $C_{n,p,q}>0$ (independent of $\delta$) such that
$$
\norm{(H_\delta-z)^{-1}}_{L^{p,2}(\R^n)\to L^{q,2}(\R^n)}
\le C_{n,p,q} \delta^{-2}|z|^{\frac n2\left(\frac1p-\frac1q-\frac 2n\right)}
,\ z\in \C\setminus[0,\infty). 
$$
Furthermore, if $(1/p,1/p')\in \overline{CG}$  for $n\ge4$ or $(1/p,1/p')\in \overline{CG}\setminus\{G\}$ for $n=3$ then 
$$
\norm{(H_\delta-z)^{-1}}_{L^{p,2}_rL^2_\omega(\R^n)\to L^{p',2}_rL^2_\omega(\R^n)}
\le C_{n,p} \delta^{-2}|z|^{\frac n2\left(\frac1p-\frac1q-\frac 2n\right)}
,\quad z\in \C\setminus[0,\infty). 
$$
\end{theorem}

Note that, in contrast to the critical case, the estimates for $p=2n/(n+2)$ hold.  

\begin{proof}
Since $(\frac 1p,\frac{n-2}{2n}),(\frac{1}{q'},\frac{n-2}{2n})\in \Omega_0$ and $rV\in L^{n,\infty}$, Proposition \ref{proposition_free_1} (1) implies
 \begin{align*}
\norm{|x|V_\delta(-\Delta-z)^{-1}f}_{L^2}
&\le C|z|^{\frac n2\left(\frac1s-\frac{n-2}{2n}-\frac 2n\right)}\norm{rV_\delta}_{L^{n,\infty}}\norm{f}_{L^{s,2}},\ s\in \{p,q'\},\\ 
\norm{|x|^{-1}(-\Delta-z)^{-1}f}_{L^2}&\le C|z|^{\frac n2\left(1-\frac1q-\frac{n-2}{2n}-\frac 2n\right)}\norm{f}_{L^{q',2}}. 
\end{align*}
These three estimates, together with \eqref{proposition_free_1_1} and \eqref{proposition_Sobolev_1_2}, allow us to use Lemma \ref{lemma_abstract_1} with the choice $T_0=-\Delta$, $T=H_\delta$, $\A=L^{p,2}$, $\B=L^{q',2}$, $Y=|x|V_\delta$, $Z=|x|^{-1}$ and $W=\Id$ to conclude the first statement. Using \eqref{proposition_free_1_2} instead of \eqref{proposition_free_1_1}, we similarly obtain the second statement.
\end{proof}
As in the critical case, this theorem implies

\begin{corollary}
\label{corollary_Sobolev_8}
Let $n\ge3$, $\gamma>0$ and $1/p_\gamma=1/(n+2\gamma)+1/2$. Then
\begin{align*}
\norm{(H_\delta-z)^{-1}}_{L^{p_\gamma,2}(\R^n)\to L^{p_\gamma',2}(\R^n)}
&\le C_{\gamma,n} \delta^{-\min(n+2\gamma,n+1)} 
K_\gamma(z)^{-\frac{1}{n/2+\gamma}},\\
\norm{(H_\delta-z)^{-1}}_{L^{p_\gamma,2}_rL^2_\omega(\R^n)\to L^{p_\gamma',2}_rL^2_\omega(\R^n)}
&\le C_{\gamma,n,\ep}\delta^{-\min\left(n+2\gamma,\frac{n(n+1)}{n-1}+\ep\right )} 
K_{\gamma,\ep}^{\mathrm{rad}}(z)^{-\frac{1}{n/2+\gamma}}
\end{align*}
for all $z\in \C\setminus[0,\infty)$, where $\ep=0$ if $n\ge4$ and $\ep>0$ if $n=3$. 
\end{corollary}

\section{Proof of Theorems \ref{theorem_1} and \ref{theorem_2}}
\label{section_proof}

Here we prove the main theorems. The proof is based on the Birman-Schwinger principle and the following estimates for the Birman-Schwinger kernel $V_1(H_\delta-z)^{-1}V_2$, where $V_1=|V|^\frac12$ and $|V_2|=V_1\sgn V$ so that $V=V_1V_2$. Here $\sgn V$ is the complex sign function of $V$. 

\begin{corollary}	
\label{corollary_3}
Let $n\ge3$, $\gamma>0$, $z\in \C\setminus[0,\infty)$, $\ep=0$ if $n\ge4$ and $\ep>0$ if $n=3$. Then there exist $\wtilde C_{\gamma,n},\wtilde C_{\gamma,n,\ep}>0$ (independent of $V$ and $z$) such that the following statements are satisfied:\\
{\rm (1)} The critical case:  Let $H_0=-\Delta-C_{\Hardy}|x|^{-2}$. Then
\begin{align}
\label{corollary_3_1}
\norm{V_1(H_0-z)^{-1}V_2}_{L^2(\R^n)\to L^2(\R^n)}
&\le \wtilde C_{\gamma,n}
K_\gamma(z)^{-\frac{1}{n/2+\gamma}}
\norm{V}_{L^{\frac n2+\gamma,\infty}(\R^n)},\\
\label{corollary_3_2}
\norm{V_1(H_0-z)^{-1}V_2}_{L^2(\R^n) \to L^2(\R^n)}
&\le \wtilde C_{\gamma,n,\ep}
K_{\gamma,\ep}^{\mathrm{rad}}(z)^{-\frac{1}{n/2+\gamma}}
\norm{V}_{L^{\frac n2+\gamma,\infty}_rL^\infty_\omega(\R^n)}.
\end{align}
{\rm (2)} The subcritical case: Let $\delta>0$ and $H_\delta=-\Delta+V_\delta$. Then  \eqref{corollary_3_1} and \eqref{corollary_3_2} hold with $H_0$, $\wtilde C_{\gamma,n}$ and $\wtilde C_{\gamma,n,\ep}$ replaced by $H_\delta$, $\wtilde C_{\gamma,n}\delta^{-\min(n+2\gamma,n+1)}$ and $\wtilde C_{\gamma,n,\ep}\delta^{-\min\left(n+2\gamma,\frac{n(n+1)}{n-1}+\ep\right)}$, respectively. 
\end{corollary}

\begin{proof}
Taking the facts that $1/p_\gamma=1/(n+2\gamma)+1/2$ and $V_1,V_2\in L^{n+2\gamma,\infty}$ into account, the assertion clearly follows from Corollaries \ref{corollary_Sobolev_2} and \ref{corollary_Sobolev_8} and H\"older's inequality. 
\end{proof}

\begin{proof}[Proof of Theorems \ref{theorem_1} and \ref{theorem_2}]
We shall show \eqref{theorem_1_1} only, proofs of other cases being analogous. Since both  $V_1(H_0+1)^{-\frac12}$ and $V_2(H_0+1)^{-\frac12}$ are compact, the Birman-Schwinger principle (see \cite[Section 4]{Fra3}) asserts that $E\in \C\setminus[0,\infty)$ is an eigenvalue of $H_0+V$ (which means $E\in \sigma_{\mathrm{d}}(H_0+V)$ in the present case) if and only if $-1$ is an eigenvalue of $V_1(H_0-E)^{-1}V_2$. In this case we have  $\norm{V_1(H_0-E)^{-1}V_2}_{L^2\to L^2}\ge1$. Therefore \eqref{theorem_1_1} follows from \eqref{corollary_3_1}. 
\end{proof}

\appendix

\section{$m$-Sectorial operators and basic spectral properties}
\label{appendix_form_compact}
Here we provides the precise definition of operators $H_0+V$ and $H_\delta+V$ and their basic spectral properties in an abstract setting. As references we mention monographs \cite{Dav,Kat}.  

Let $\H$ be a complex Hilbert space with inner product $\<\cdot,\cdot\>$ and norm $\norm{\cdot}$. We first recall the notion of sectorial forms  and $m$-sectorial operators: 

\begin{definition}[Sectorial form]
\label{sectorial_form}
A quadratic form $(q,D(q))$ on $\H$ is said to be sectorial if there exist $c\in \R$ and $\theta\in[0,\pi/2)$ such that its numerical range $\mathrm{Num}(q):=\{q(u)\ |\ u\in D(q),\ \norm{u}=1\}$ is contained in a sector $\Gamma_{c,\theta}:=\{z\in \C\ |\ |\arg(z-c)|\le\theta\}$. 
\end{definition}

\begin{definition}[$m$-Sectorial operator]
\label{sectorial_operator}
A closed linear operator $(T,D(T))$ on $\H$ is said to be $m$-sectorial if its numerical range $\mathrm{Num}(T):=\{\<Tu,u\>\ |\ u\in D(T),\ \norm{u}=1\}$  is contained in a sector $\Gamma_{c,\theta}$ for some $c\in \R$ and $\theta\in[0,\pi/2)$ and, for $z\in \C$ with $\Re z<c$, $T-z$ is invertible on $\H$ and $\norm{(T-z)^{-1}}_{\H\to \H}\le |\Re z|^{-1}$. 
\end{definition}

The following lemma provides the precise definition of operators $H_0+V$ and $H_\delta+V$:

\begin{lemma}
\label{lemma_sectorial_1}
Let $(T_0,D(T_0))$ be a non-negative self-adjoint operator on $\H$ and $V_1,V_2$ densely defined closed operators on $\H$ such that $D(T_0^{1/2})\subset D(V_j)$ and $V_j(T_0+1)^{-1/2}$ are compact for $j=1,2$. Then the quadratic form $
q(u)=\<T_0u,u\>+\<V_1u,V_2u\>
$ with $D(q)=D(T_0^{1/2})$ 
generates an $m$-sectorial operator $T$ with form domain $D(q)$ such that its operator domain $D(T)$ 
is a dense linear subspace of $D(q)$. 
\end{lemma}

\begin{proof}
We follow \cite[Lemma B.1]{Fra3}. Since $V_j(T_0+1)^{-1/2}$ are compact, it is seen from  a standard approximation argument that, for any $\ep>0$, there exist finite rank operators $B_j$ and remainders $R_j$ such that $V_j(T_0+1)^{-1/2}=B_j+R_j$, $B_j(T_0+1)^{1/2}$ are bounded and $\norm{R_j}\le \ep$. Then we have 
\begin{align*}
|\<V_1u,V_2u\>|
&\le (\norm{B_j(T_0+1)^{1/2}u}+\norm{R_j(T_0+1)^{1/2}u})^2\\
&\le C_\ep\norm{u}^2+\ep^2\norm{(T_0+1)^{1/2}u}^2\\
&\le C_\ep\norm{u}^2+\ep^2\norm{T_0^{1/2}u}^2
\end{align*}
for all $u\in D(T_0^{1/2})$, which implies that $\Re q(u)$ is lower semi-bounded, \emph{i.e.}, $\Re q(u)\ge -c\norm{u}$ with some $c\ge0$. Moreover, if we equip $D(q)$ with a norm $\norm{u}_{+1}=(\Re q(u)+(c+1)\norm{u}^2)^{1/2}$ then $(D(q),\norm{\cdot}_{+1})$ becomes a Hilbert space and $q$ is bounded on $(D(q),\norm{\cdot}_{+1})$. Now it follows from \cite[Theorems 1.33 and 3.4 in Chapter IV]{Kat} that $q$ is sectorial and generates an $m$-sectorial operator $T$ such that $D(T)$ satisfies the above property. 
\end{proof}

Now we consider basic spectral properties of $T$. At first note that, for any linear closed operator $A$, if  $\C\setminus\overline{\Num(A)}$ is connected and contains at least one point in $\rho(A)$, then $\sigma(A)\subset \Num(A)$ (see, \emph{e.g.}, \cite[Lemma 9.3.14]{Dav}). In particular, we have $\sigma(T)\subset \Gamma_{c,\theta}$ since $\{z\ |\ \Re z<c\}\subset \rho(T)$. 
Recall that the discrete spectrum $\sigma_{\mathrm{d}}(T)$ and the essential spectrum $\sigma_{\mathrm{ess}}(T)$ are defined by
\begin{align*}
\sigma_{\mathrm{d}}(T)&=\{\lambda\in \sigma(T)\ |\ \text{$\lambda$ is isolated and algebraic multiplicity of $\lambda$ is finite}\},\\
\sigma_{\ess}(T)&=\{\lambda\in \sigma(T)\ |\ \text{$T-\lambda$ is not a Fredholm operator}\}, 
\end{align*}
where a closed operator $A$ is said to be Fredholm if $\Ran A$ is closed and  $\dim(\Ker A)<\infty$ and $\codim(\Ran A)<\infty$. Then we have 

\begin{lemma}[{\cite[Lemma B.2]{Fra3}}]
\label{lemma_sectorial_2}
Under conditions in Lemma \ref{lemma_sectorial_1}, $\sigma_{\ess}(T)=\sigma_{\ess}(T_0)$. Furthermore, $\sigma(T)=\sigma_{\mathrm{d}}(T)\cup\sigma_{\mathrm{ess}}(T)$ and $\sigma_{\mathrm{d}}(T)\cap\sigma_{\mathrm{ess}}(T)=\emptyset$. 
\end{lemma}

\section{Real interpolation and Lorentz spaces}
\label{appendix_interpolation}
Here a brief summery of real interpolation spaces and Lorentz spaces is given without proofs. One can find a much more detailed exposition in \cite{BeLo,Gra,Tri}. 

\subsection{Real interpolation}
Given a Banach couple $(\A_0,\A_1)$ and $0<\theta<1$ and $1\le q\le \infty$, one can define a Banach space $\A_{\theta,q}=(\A_0,\A_1)_{\theta,q}$ by the so-called $K$-method, which satisfies following properties: 
\begin{proposition}[{\cite[Theorem 3.4.1]{BeLo}}]
\label{proposition_interpolation_1}
Let $0<\theta<1$ and $1\le q\le\infty$. Then\\
{\rm(1)} $(\A_0,\A_0)_{\theta,q}=\A_0$  and $(\A_0,\A_1)_{\theta,q}=(\A_1,\A_0)_{1-\theta,q}$ with equivalent norms.\\
{\rm(2)} If $1\le q_1\le q_2\le \infty$ then $(\A_0,\A_1)_{\theta,1}\subset (\A_0,\A_1)_{\theta,q_1}\subset (\A_0,\A_1)_{\theta,q_2}\subset (\A_0,\A_1)_{\theta,\infty}$. 
\end{proposition}

\begin{theorem}[{\cite[Theorem 3.1.2]{BeLo}}]
\label{theorem_interpolation_2}
Let $(\A_0,\A_1)$ and $(\B_0,\B_1)$ be two Banach couples, $0<\theta<1$ and $1\le q\le\infty$. Suppose  that $T$ is a bounded linear operator from $(\A_0,\A_1)$ to $(\B_0,\B_1)$ in the sense  that $T:\A_j\to \B_j$ and
\begin{align*}
\norm{Tf}_{\B_j}&\le M_j\norm{f}_{\A_j},\quad f\in \A_j,\ j=0,1.
\end{align*}
Then $T$ is bounded from $\A_{\theta,q}$ to $\B_{\theta,q}$ and satisfies
$$
\norm{Tf}_{\B_{\theta,q}}\le M_0^{1-\theta}M_1^\theta\norm{f}_{\A_{\theta,q}},\quad f\in \A_{\theta,q}.
$$
In particular, $\norm{T}_{\A_{\theta,q}\to \B_{\theta,q}}$ is uniformly bounded in $q$. 
\end{theorem}

\subsection{Lorentz spaces}

Let $(X,\mu)$ be a $\sigma$-finite measure space and $\A$ a Banach space. We denote by $L^p\A=L^p(X, \mu; \A)$ the Bochner-Lebesgue space equipped with norm $\norm{f}_{L^p\A}:=\norm{\norm{f}_{\A}}_{L^p(X)}$. Given a strongly $\mu$-measurable function $f:X\to \A$, we let $\mu_f(\alpha)=\mu(\{x\ |\ \norm{f}_{\A}>\alpha\})$.
If we define the decreasing rearrangement of $f$ by $f^{*}(t) =\inf\{\alpha\ |\ \mu_f(\alpha)\le t\}$ then the Bochner-Lorentz space $L^{p,q}\A=L^{p,q}(X,\mu;\A)$ is the set of strongly $\mu$-measurable $f:X\to\A$ such that the quasi-norm 
$$
\norm{f}^*_{L^{p,q}\A}:=\norm{t^{1/p-1/q}f^*(t)}_{L^q(\R_+,dt)}=p^{1/q}\norm{\alpha \mu_f(\alpha)^{1/p}}_{L^q(\R_+,\alpha^{-1}d\alpha)}
$$
is finite. Although $\norm{f}^*_{L^{p,q}\A}$ is not a norm in general, if $1<p<\infty$ and $1\le q\le \infty$ (which are sufficient for our purpose), then there is a norm defined by 
$$
\norm{f}_{L^{p,q}\A}:=\norm{f^{**}}_{L^{p,q}\A}^*,\quad f^{**}(t):=\frac1t\int_0^tf^*(\alpha)d\alpha,
$$
which makes $L^{p,q}\A$ a Banach space. Furthermore, $\norm{\cdot}_{L^{p,q}\A}$ is equivalent to $\norm{\cdot}^*_{L^{p,q}\A}$ in the sense that $\norm{f}^*_{L^{p,q}\A}\le \norm{f}_{L^{p,q}\A}\le C(p,q)\norm{f}^*_{L^{p,q}\A}$ with some constant $C(p,q)>0$. Thus all continuity estimates for linear operators can be expressed in terms of $\norm{\cdot}_{L^{p,q}\A}^*$. $L^{p,q}\A$ is increasing in $q$: $L^{p,1}\A\subset L^{p,q_1}\A\subset L^{p,p}\A=L^p\A\subset L^{p,q_2}\A\subset L^{p,\infty}\A$ if $1<q_1<p<q_2<\infty$. Moreover, $L^{p,q}\A$ is characterized by the real interpolation of Bochner-Lebesgue spaces:

\begin{theorem}	[{\cite[Theorem 2 in pages 134]{Tri}}]
\label{theorem_Lorentz_2}
For $0<\theta<1$, $1<p_1<p_2<\infty$ with $\frac1p=\frac{1-\theta}{p_1}+\frac{\theta}{p_2}$ and $1\le q\le \infty$, one has $
(L^{p_0}\A,L^{p_2}\A)_{\theta,q}=L^{p,q}\A
$ with equivalent norms.
\end{theorem}

If $X$ is non-atomic, $\A=\C$ and $1<p,q<\infty$ then $L^{p,q}(X;\C)'=L^{p',q'}(X;\C)$, where $r'=r/(r-1)$ is the H\"older conjugate of $r$. Furthermore, when $\A$ is also a Lorentz space, a natural characterization of dual spaces holds:
\begin{proposition}[{\cite[Proposition 6.3]{Fer}}]
\label{proposition_Lorentz_1}
Let $(X,\mu),(Y,\nu)$ be two non-atomic $\sigma$-finite measure spaces, $1<p_1,p_2<\infty$ and  $1\le q_1,q_2\le\infty$. Then $(L^{p_1,q_1}(X;L^{p_2,q_2}(Y)))^*=L^{p_1',q_1'}(X;L^{p_2',q_2'}(Y))$. 
\end{proposition}

Finally we record some properties used throughout this paper  for the case when $\A=\C$: 
\begin{itemize}
\item Scaling invariance: $\norm{f(\lambda x)}_{L^{p,q}}=\lambda^{-n/p}\norm{f}_{L^{p,q}}$ for $\lambda>0$ and $\norm{|f|^s}_{L^{p,q}}=\norm{f}_{L^{ps,qs}}^s$. 
\item H\"older's inequality: if $1\le p,q,p_j,q_j\le \infty$ such that $\frac1p=\frac{1}{p_1}+\frac{1}{p_2}$ and $\frac1q=\frac{1}{q_1}+\frac{1}{q_2}$ then
$$
\norm{fg}_{L^{p,q}(X)}\le C\norm{f}_{L^{p_1,q_1}}\norm{g}_{L^{p_2,q_2(X)}}.
$$
\item Sobolev's inequality (see \cite[Lemma 32.1]{Tar}): if $X=\R^n$ and $n\ge3$ then 
$$
\norm{f}_{L^{2n/(n-2),2}(\R^n)}\le C\norm{\nabla f}_{L^2(\R^n)}.
$$
Furthermore, if $0<s<n/2$ then $\H^s(\R^n)\,\hookr\,L^{2n/(n-s),2}(\R^n)$. 
\end{itemize}

\section{Relative form compactness of potentials}
\label{appendix_remark}
Here we show that if $V\in L^{p,\infty}_0(\R^n)+L^\infty_0(\R^n)$ with $p>n/2$ then both $|V|^{1/2}(1-\Delta)^{-1/2}$ and $|V|^{1/2}(H_0+1)^{-1/2}$ are compact. We begin with two basic lemmas. 
\begin{lemma}
\label{lemma_form_compact_1}
Let $n\ge3$ and $s<1$. Then $\<D\>^s(H_0+1)^{-1/2}$ is bounded on $L^2(\R^n)$.
\end{lemma}

\begin{proof}
Assume $0\le s<1$ without loss of generality. It follows from \cite[Theorem 1.2]{Fra1} that 
$$
\norm{(-\Delta)^{s/2}f}_{L^2}\le C_s\norm{H_0^{1/2}f}_{L^2}^s\norm{f}^{1-s}_{L^2}\le C'_s(\norm{H_0^{1/2}f}_{L^2}+\norm{f}_{L^2}),\quad f\in C_0^\infty(\R^n),
$$
which, together with the density of $C_0^\infty(\R^n)$ in $D(H_0^{1/2})$, implies the assertion. 
\end{proof}

\begin{lemma}
\label{lemma_form_compact_2}
For any $f,g\in L^\infty_0(\R^n)$, $f(x)g(D)$ is compact on $L^2(\R^n)$. 
\end{lemma}

\begin{proof}
Define $f_R=f\mathds1_{\{|x|\le R\}}$ and $g_R=g\mathds1_{\{|\xi|\le R\}}$ for $R>0$. For each $R>0$, $f_R(x)g_R(D)$ is of Hilbert-Schmidt class (and thus compact) since $f_R(x)\check g_R(x-y)\in L^2(\R^{2n})$. Furthermore, since 
\begin{align*}
&\norm{f(x)g(D)-f_R(x)g_R(D)}_{L^2\to L^2}\\
&\le \norm{f}_{L^\infty}\norm{g\mathds1_{\{|\xi|> R\}}}_{L^\infty}+\norm{f\mathds1_{\{|x|> R\}}}_{L^\infty}\norm{g_R}_{L^\infty}\to0
\end{align*}
as $R\to0$ by the hypothesis $f,g\in L^\infty_0$, $f(x)g(D)$ is also compact. 
\end{proof}

\begin{lemma}
\label{lemma_form_compact_3}
Let $n\ge3$ and $w\in L^{p,\infty}_0(\R^n)+L^\infty_0(\R^n)$ with $p>n$. Then, for any $s>n/p$, $w(x)\<D\>^{-s}$ is compact. In particular, both $w(x)(-\Delta+1)^{-1/2}$ and $w(x)(H_0+1)^{-1/2}$ are compact. 
\end{lemma}

\begin{proof}
We may assume $w\in L^{p,\infty}_0$ with $n<p<\infty$, the proof for the case $w\in L^\infty_0$ being similar. H\"older's inequality and Sobolev's inequality in Lorentz spaces yield 
\begin{align}
\label{lemma_form_compact_3_1}
\norm{w(x)\<D\>^{-n/p}}_{L^2\to L^2}\le C\norm{w}_{L^{p,\infty}}
\end{align}
Define $w_R:=\mathds 1_{\{|x|\le R\}}w$. As in Lemma \ref{lemma_form_compact_2}, it suffices to show  that $w_R\<D\>^{-s}$ is compact. Take $\chi\in C_0^\infty(\R^n)$ with $\chi\equiv 1$ for $|x|\le R$ and decompose $w_R\<D\>^{-s}$ as $$w_R\<D\>^{-s}=w_R\chi\<D\>^{-s}=w_R\<D\>^{-n/p}\cdot\<D\>^{n/p}\chi\<D\>^{-n/p}\<x\>^{\ep}\cdot\<x\>^{-\ep}\<D\>^{n/p-s}$$ with some $\ep>0$. Here $\<D\>^{n/p}\chi\<D\>^{-n/p}\<x\>^{\ep}$ is a pseudodifferential operator with a smooth symbol of order zero and hence bounded on $L^2$ by the Calderon-Vaillancourt theorem (see, {\it e.g.}, \cite[Chapter 2]{Mar}). Furthermore, $\<x\>^{-\ep}\<D\>^{n/p-s}$ is compact by Lemma \ref{lemma_form_compact_2}. These two facts and \eqref{lemma_form_compact_3_1} imply $w_R\<D\>^{-s}$ is compact. Finally, the compactness of $w(H_0+1)^{-1/2}$ follows from the compactness of $w\<D\>^{-s}$ and Lemma \ref{lemma_form_compact_1} by taking $n/p<s<1$. 
\end{proof}

\section{Proof of Proposition \ref{proposition_Sobolev_1}}
\label{appendix_resolvent}
In what follows the notation $r=|x|$ as well as $\partial_r=r^{-1}x\cdot\nabla$ will be used frequently. 
The proof of Proposition \ref{proposition_Sobolev_1} is essentially same as that of \cite[Theorem 1.6]{BVZ} (see also \cite[Appendix B]{BoMi}). However, since only the subcritical case was considered in these papers, we give the details of the proof in the critical case. We set $V_0=-C_{\Hardy}|x|^{-2}$ for short. Let $f\in C_0^\infty(\R^n)$ be such that $Pf\equiv0$ and $\lambda+i\ep\in \C\setminus[0,\infty)$ with $\lambda,\ep\in\R$. Consider the Helmholtz equation
\begin{align}
\label{proof_D_1}
(H_0-\lambda- i\ep)u=f,
\end{align} 
which has a solution $u=(H_0-\lambda-i\ep)^{-1}f\in D(H)$. Note that $u$ also satisfies $Pu\equiv0$ since $(H_0-z)^{-1}$ and $P$ commute. We  consider the case $\ep\ge0$ only, the proof for the case $\ep<0$ being analogous. We prepare three key lemmas:

\begin{lemma}
\label{lemma_Hardy}
{\rm (1)} An improved Hardy inequality: if $f\in \H^1(\R^n)$ and $Pf\equiv0$ then
$$
\frac{n^2}{4}\norm{r^{-1}f}_{L^2(\R^n)}^2\le \norm{\nabla f}_{L^2(\R^n)}^2. 
$$
{\rm (2)} A weighted Hardy inequality: if $f\in \H^1(\R^n)$ and $r^{1/2}\nabla f\in L^2(\R^n)$ then
$$
\frac{(n-1)^2}{4}\norm{r^{-1/2}f}_{L^2(\R^n)}^2\le \norm{r^{1/2}\nabla f}_{L^2(\R^n)}^2. 
$$
\end{lemma}

\begin{proof}
We refer to  \cite[Lemma 2.4]{EkFr} for (1) and \cite[Proposition 8.1]{MSS} for (2) in which simple proofs can be found. 
\end{proof}

\begin{remark}
We learn by Lemma \ref{lemma_Hardy} (1) that $\norm{\nabla g}_{L^2}^2-C_{\Hardy}\norm{|x|^{-1}g}_{L^2}^2\sim \norm{\nabla g}_{L^2}^2$ if $g\in C_0^\infty(\R^n)$ and $Pg\equiv0$, which implies that any  $g\in D(q_0)$ satisfying $Pg\equiv0$ belongs to $\H^1$. In particular, if $f\in C_0^\infty(\R^n)$ and $Pf\equiv0$ then $u=(H_0-\lambda-i\ep)^{-1}f$ belongs to $\H^1$. This observation is useful to justify the computations in the proof of the next lemma. 
\end{remark}

\begin{lemma}
\label{lemma_appendix_D_1}
The following five identities hold:
\begin{align}
\label{proof_D_2}
\int\Big(|\nabla u|^2+V_0|u|^2-\lambda |u|^2\Big)dx
&=\Re\int f\overline{u}dx,\\
\label{proof_D_3}
-\ep\int|u|^2dx
&=\Im \int f\overline{u}dx,\\
\label{proof_D_4}
\int \Big(r|\nabla u|^2-\lambda r|u|^2+rV_0|u|^2+\Re(\overline u\partial_ru)\Big)dx
&=\Re \int rf\overline udx,\\
\label{proof_D_5}
\int\Big(-\ep r|u|^2+\Im(\overline u\partial_ru)\Big)dx
&=\Im \int rf\overline udx,\\
\label{proof_D_6}
\int\Big(2|\nabla u|^2-(r\partial_rV_0)|u|^2-2\ep \Im(\overline ur\partial_ru)\Big)dx
&=\Re \int f(2r\partial_r \overline u+n\overline u)dx. 
\end{align}
Furthermore, we have $r^{1/2}\nabla u\in L^2$, $r^{1/2}u\in L^2$, $r^{1/2}u\in \H^1$. 
\end{lemma}

\begin{proof}We only outline the proof and refer to \cite[Appendix B]{BoMi} for more details. 
Note that conditions $rV_0\in L^{n,\infty}$ and $r\partial_r V_0\in L^{n/2,\infty}$ are enough to justify following computations. 

\eqref{proof_D_2} and \eqref{proof_D_3} are verified by multiplying \eqref{proof_D_1} by $\overline u$, integrating over $\R^n$ and taking the real and imaginary parts. \eqref{proof_D_4} and \eqref{proof_D_5} follow from multiplying \eqref{proof_D_1} by $r\overline u$, integrating over $\R^n$ and taking the real and imaginary parts. 
\eqref{proof_D_6} can be seen from multiplying \eqref{proof_D_1} by $iA\overline u$ with $iA=r\partial_r+n/2$, integrating over $\R^n$ and taking the real part. 

By \eqref{proof_D_5} and \eqref{proof_D_4}, we have $r^{1/2}\nabla u\in L^2$ and $r^{1/2}u\in L^2$. These two properties, together with Lemma \ref{lemma_Hardy} (2), imply $r^{1/2}u\in \H^1$. 
\end{proof}

\begin{lemma}
\label{lemma_appendix_D_2}
Let $0<\ep<\lambda$ and $v_\lambda=e^{-i\lambda^{\frac12}r}u$. Then one has
\begin{equation}
\begin{aligned}
\label{proof_D_7}
&\int \Big(|\nabla v_\lambda|^2-\partial_r(rV_0)|v_\lambda|^2+\ep\lambda^{-\frac12}r|\nabla v_\lambda|^2+\ep\lambda^{-\frac12}rV_0|v_\lambda|^2\Big)dx\\
&=\Re\int\Big(-\ep\lambda^{-\frac12}\overline ue^{i\lambda r}\partial_rv_\lambda+(n-1)f\overline u+\ep\lambda^{-\frac12}rf\overline u+2rf\overline{e^{i\lambda^{\frac12}r}\partial_rv_\lambda}\Big)dx. 
\end{aligned}
\end{equation}
\end{lemma}

\begin{proof}
The formula \eqref{proof_D_7} is derived by computing
$\eqref{proof_D_6}-\eqref{proof_D_2}-2\lambda^{\frac12}\times\eqref{proof_D_5}+\ep\lambda^{-\frac12}\times\eqref{proof_D_4}.$ (see  \cite[Section 2]{BVZ} and \cite[Appendix B]{BoMi} for more details). 
\end{proof}

\begin{remark}
Lemmas \ref{lemma_appendix_D_1} and \ref{lemma_appendix_D_2} also hold for the subcritical case (with $V_0$ replaced by $V_\delta$) without assuming that $Pf\equiv0$. Indeed, since $D(q_\delta)=\H^1$ in the subcritical case, we have $(H_\delta-\lambda-i\ep)^{-1}f\in \H^1$ and the proof of Lemma \ref{lemma_appendix_D_1} thus works for any $f\in C_0^\infty(\R^n)$. 
\end{remark}

\begin{proof}[Proof of Proposition \ref{proposition_Sobolev_1} (1)]
Let $f\in C_0^\infty(\R^n)$ be such that $Pf\equiv0$. Note that $Pu=Pv_\lambda=0$ since $V_0$ is radially symmetric. It suffices to show
\begin{align}
\label{proof_D_8}
\norm{r^{-1}u}_{L^2}\le C\norm{rf}_{L^2}
\end{align}
uniformly in $\lambda\in\R$ and $\ep>0$ or in $\lambda<0$ and $\ep=0$. When $\ep\ge\lambda$, \eqref{proof_D_2} and \eqref{proof_D_3} imply
\begin{equation}
\begin{aligned}
\label{proof_D_9}
\int\Big(|\nabla u|^2+V_0|u|^2\Big)dx
&\le (1+\lambda_+/\ep)\int |fu|dx\\
&\le \delta_1\norm{r^{-1}u}_{L^2}^2+\delta_1^{-1}\norm{rf}_{L^2}^2
\end{aligned}
\end{equation}
for any $\delta_1>0$, where $\lambda_+=\max\{0,\lambda\}$. Note that if $\lambda< 0$ and $\ep=0$, \eqref{proof_D_9} still holds with $1+\lambda_+/\ep$ replaced by $1$. Having the fact $Pu=0$ in mind, Lemma \ref{lemma_Hardy} (1) shows
\begin{align}
\label{proof_D_10}
\int\Big(|\nabla u|^2+V_0|u|^2\Big)dx
\ge (n-1)\norm{r^{-1}u}_{L^2}^2.
\end{align}
Taking $\delta_1=(n-1)/2$ we obtain \eqref{proof_D_8}.  

We next let $\ep<\lambda$. By Lemma \ref{lemma_Hardy} (1) and (2), the left hand side of \eqref{proof_D_7} satisfies 
\begin{equation}
\begin{aligned}
\label{proof_D_11}
&\int \Big(|\nabla v_\lambda|^2-\partial_r(rV_0)|v_\lambda|^2+\ep\lambda^{-\frac12}r|\nabla v_\lambda|^2+\ep\lambda^{-\frac12}rV_0|v_\lambda|^2\Big)dx\\
&\ge 
(n-1)\norm{r^{-1}u}_{L^2}^2+\frac{2n-3}{4}\ep\lambda^{-\frac12}\norm{r^{-1/2}u}_{L^2}^2. 
\end{aligned}
\end{equation}
Hence it suffices to show that there exist $\delta_0<(n-1)$ and $C>0$, independent of $\ep$ and $\lambda$, such that the right hand side of \eqref{proof_D_7} is bounded from above by $\delta_0\norm{r^{-1}u}_{L^2}^2+C\norm{rf}_{L^2}^2$. For the first term of the right hand side of \eqref{proof_D_7}, the Cauchy-Schwarz inequality and the classical Hardy inequality yield that there exists $C_1>0$ independent of $\ep$ and $\lambda$ such that
\begin{align*}
\ep\lambda^{-\frac12}\Big|\Re \int e^{i\lambda^{\frac12} r}(\partial_rv_\lambda) \overline udx\Big|
&\le \sqrt \ep \norm{\nabla v_\lambda}_{L^2}\norm{u}_{L^2}\\
&\le \delta_1\norm{r^{-1}u}_{L^2}^2+C_1\delta_1^{-1}\ep\norm{u}_{L^2}^2. 
\end{align*}
for any $\delta_1>0$. Here \eqref{proof_D_3} and Hardy's inequality imply
\begin{align*}
\ep\norm{u}_{L^2}^2
\le \int |fu|dx
\le \delta_1^2\norm{r^{-1}u}_{L^2}^2 + C_1'\delta_1^{-2}\norm{rf}_{L^2}^2.
\end{align*}
with some universal constant $C_1'>0$. Hence we obtain
\begin{align}
\label{proof_D_13}
\ep\lambda^{-\frac12}\Big|\Re \int e^{i\lambda^{\frac12} r}(\partial_rv_\lambda) \overline udx\Big|
\le (1+C_1)\delta_1\norm{r^{-1}u}_{L^2}^2+C_1C_1'\delta_1^{-3}\norm{rf}_{L^2}^2.
\end{align}
For other terms, similar computations yield
\begin{align}
\label{proof_D_14}
\Big|(n-1)\Re\int f\overline udx\Big|
&\le \delta_1\norm{r^{-1}u}_{L^2}^2+C_2\delta_1^{-1}\norm{rf}_{L^2}^2,\\
\Big|2\Re \int rf\overline{e^{i\lambda^{\frac12}r}\partial_rv_\lambda}dx\Big|
&\le \delta_1\norm{r^{-1}u}_{L^2}^2+C_3\delta_1^{-1}\norm{rf}_{L^2}^2,\\
\label{proof_D_15}
\ep\lambda^{-\frac12}\Big|\int rf\overline udx\Big|
\le \sqrt\ep\norm{rf}_{L^2}\norm{u}_{L^2}^2
&\le \delta_1\norm{r^{-1}u}_{L^2}^2+C_4\delta_1^{-3}\norm{rf}_{L^2}^2. 
\end{align}
\eqref{proof_D_13}--\eqref{proof_D_15} show that the right hand side of \eqref{proof_D_7} is bounded from above by $\delta_0\norm{r^{-1}u}_{L^2}^2+C\norm{rf}_{L^2}^2$ with $\delta_0=(4+C_1)\delta_1$ and $C=(C_1C_1'+C_2+C_3+C_4)\delta_1^{-3}$. Choosing $\delta_1$ so that $(4+C_1)\delta_1<n-1$ we obtain \eqref{proof_D_8}. 
\end{proof}

\begin{proof}[Proof of Proposition \ref{proposition_Sobolev_1} (2)]
Let $f\in C_0^\infty(\R^n)$ (note that $f$ does not have to satisfy $Pf\equiv0$). The most part of the proof is same as above. When $\ep\ge\lambda$, \eqref{assumption_A_1} implies 
$$
\int \Big(|\nabla u|^2+V_\delta|u|^2\Big)dx\ge \delta\norm{\nabla u}_{L^2}^2.
$$
Therefore, taking $\delta_1=\delta/2$ we obtain 
$
\norm{r^{-1}u}_{L^2}^2\le C\delta^{-2}\norm{rf}_{L^2}^2
$ for $u=(H_\delta-z)^{-1}f$. 

When $\ep<\lambda$, the only difference from the critical case is that, instead of Lemma \ref{lemma_Hardy}, we use \eqref{assumption_A_1} and \eqref{assumption_A_2} to deal with two terms $\partial_r(rV_\delta)|v_\lambda|^2$ and $\ep\lambda^{-1/2}rV_\delta|v_\lambda|^2$ as follows. For the first term, we simply use \eqref{assumption_A_2} to obtain
\begin{align}
\label{proof_D_17}
\int (|\nabla v_\lambda|^2-\partial_r(rV_\delta)|v_\lambda|^2)dx\ge \delta\norm{\nabla v_\lambda}_{L^2}^2.
\end{align}
For the second term, taking the fact $r^{1/2}v_\lambda\in\H^1$ into account, we learn by \eqref{assumption_A_1} that 
$$
-\ep\lambda^{-\frac12}\int rV_\delta|v_\lambda|^2dx\le \ep\lambda^{-\frac12}(1-\delta)\Big(\norm{r^\frac12\nabla v_\lambda}_{L^2}^2+\int|v_\lambda\nabla v_\lambda|dx+\frac14\norm{r^{-\frac12}v_\lambda}^2_{L^2}\Big).
$$
Let us fix $\delta_1>0$ arbitrarily. By the fact $\ep<\lambda$, the Cauchy-Schwartz and Hardy's inequalities, second and third terms of the right hand side satisfy
\begin{align*}
&(1-\delta)\ep\lambda^{-1/2}\Big(\int|v_\lambda\nabla v_\lambda|dx+\norm{r^{-\frac12}v_\lambda}_{L^2}^2\Big)\\
&\le (1-\delta)\sqrt\ep\norm{v_\lambda}_{L^2}(\norm{\nabla v_\lambda}_{L^2}+\norm{r^{-1}v_\lambda}_{L^2})\\
&\le \delta_1\norm{\nabla v_\lambda}_{L^2}^2+C_0\delta_1^{-1}\ep\norm{v_\lambda}^2,
\end{align*}
with some $C_0>0$ independent of $\ep$, $v_\lambda$, $\delta$ and $\delta_1$. 
Here \eqref{proof_D_3} and Hardy's inequality imply
\begin{align*}
\ep\norm{v_\lambda}_{L^2}^2
&\le \delta_1^2\norm{\nabla v_\lambda}_{L^2}^2+C_{\Hardy}^{-1}\delta_1^{-2}\norm{rf}_{L^2}^2,
\end{align*}
and hence, setting $C_1=C_{\Hardy}^{-1}C_0$, we obtain
\begin{equation}
\begin{aligned}
\label{proof_D_18}
&-\ep\lambda^{-\frac12}\int rV_\delta|v_\lambda|^2dx\\
&\le \ep\lambda^{-\frac12}(1-\delta)\norm{r^\frac12\nabla v_\lambda}_{L^2}^2
+(1+C_1)\delta_1\norm{\nabla v_\lambda}_{L^2}^2
+C_1\delta_1^{-3}\norm{rf}_{L^2}^2. 
\end{aligned}
\end{equation}
By virtue of \eqref{proof_D_17} and \eqref{proof_D_18}, the same argument as in the critical case  yields that if we take $\delta_1=a\delta$ with $a>0$ small enough, then $\norm{ r^{-1}u}_{L^2}^2\le C\delta^{-4}\norm{rf}_{L^2}^2$ which completes the proof. 
\end{proof}

$ $\\
{\bf Acknowledgments.} The author would like to thank Jean-Marc Bouclet for valuable discussions and for hospitality at the Institut de Math\'ermatiques de Toulouse, Universit\'e Paul Sabatier, where this work has been done. He is partially supported by JSPS Grant-in-Aid for Young Scientists (B) (No. 25800083) and by Osaka University Research Abroad Program (No. 150S007).


\end{document}